\documentclass[reqno, 11pt, oneside]{amsart}
\usepackage{amsmath}
\usepackage{csquotes}
\usepackage{amsthm}
\usepackage{graphicx}
\usepackage{amsfonts}
\usepackage{quiver}
\usepackage{bookmark}
\usepackage{xcolor}
\usepackage[english]{babel}
\usepackage[backend=biber, style=alphabetic]{biblatex}
\usepackage{graphicx}
\usepackage{fancyhdr}
\usepackage{tabularx}
\usepackage{mathtools}
\usepackage{mathrsfs}

\newcommand{\spec}{\text{Spec}}

\newcommand{\Hom}{\text{Hom}}
\newcommand{\rank}{\text{rank}}
\newcommand{\Ker}{\text{Ker}}
\newcommand{\Image}{\text{Im}}
\newcommand{\Coker}{\text{Coker}}
\newcommand{\Gal}{\text{Gal}}
\newcommand{\Aut}{\text{Aut}}

\newcommand{\Frob}{\text{Frob}}

\newcommand{\Sh}{\text{Sh}}
\newcommand{\GSp}{\text{GSp}}

\newcommand{\GL}{\text{GL}}
\newcommand{\mY}{\mathcal{Y}}

\newcommand{\mH}{\mathcal{H}}
\newcommand{\mG}{\mathcal{G}}
\newcommand{\mO}{\mathcal{O}}
\newcommand{\F}{\mathbb{F}}
\newcommand{\Q}{\mathbb{Q}}
\newcommand{\Z}{\mathbb{Z}}

\newcommand{\Lie}{\text{Lie}}

\newcommand{\mA}{\mathcal{A}}

\newcommand{\mC}{\mathcal{C}}

\newcommand{\mI}{\mathcal{I}}

\newcommand{\Fil}{\text{Fil}}
\newcommand{\mM}{\mathcal{M}}
\newcommand{\mD}{\mathcal{D}}

\newcommand{\Ext}{\text{Ext}}

\newcommand{\C}{\mathbb{C}}
\newcommand{\mE}{\mathcal{E}}

\newcommand{\mF}{\mathcal{F}}

\newcommand{\mbL}{\mathbb{L}}

\newcommand{\mX}{\mathcal{X}}

\newcommand{\mV}{\mathcal{V}}
\newcommand{\etpi}{\pi_1^{\text{\'et}}}
\newcommand{\der}{\text{der}}
\newcommand{\ad}{\text{ad}}
\newcommand{\mS}{\mathcal{S}}

\newcommand{\mbH}{\mathbb{H}}

\newcommand{\V}{\mathbb{V}}

\newcommand{\KS}{\text{KS}}
\newcommand{\Pic}{\text{Pic}}
\newcommand{\Gr}{\text{Gr}}
\newcommand{\Tr}{\text{Tr}}

\usepackage{amsthm}

\theoremstyle{definition}
\newtheorem{theorem}{Theorem}[section]
\newtheorem{lemma}[theorem]{Lemma}
\newtheorem{proposition}[theorem]{Proposition}
\newtheorem{corollary}[theorem]{Corollary}
\newtheorem{conjecture}[theorem]{Conjecture}

\theoremstyle{definition}

\theoremstyle{definition}
\newtheorem{question}[theorem]{Question}

\theoremstyle{definition}
\newtheorem{definition}[theorem]{Definition}

\newtheorem{fact}[theorem]{Fact}

\theoremstyle{remark}
\newtheorem{remark}[theorem]{Remark}

\newtheorem*{acknowledgment}{Acknowledgment}
\newtheorem*{Organization}{Organization of this paper}

\usepackage[margin=1in]{geometry}

\title{Non-liftability of Families of Abelian Varieties with Small $l$-adic local system}
\author{Haochen Cheng}

\address{Department of Mathematics, Northwestern University}

\email{haochencheng2028@u.northwestern.edu}

\date{\today}

\begin{document}

\maketitle

\begin{abstract}
  We study families of abelian varieties over smooth proper curves with small $l$-adic local system over characteristic $p$. We show that such abelian schemes have a non-nef Hodge bundle and cannot be lifted to $W_2(k)$.   We also establish an Arakelov-type inequality for families of abelian varieties over smooth proper curves in characteristic $p$, assuming $W_2(k)$-liftability. 

\end{abstract}

\tableofcontents

\section{Introduction}

Throughout this paper, we fix an algebraically closed field $k=\bar{k}$ of characteristic $p>0$, unless otherwise specified.
This paper is motivated by the following question:
\begin{question}\label{problem 1}
  Let $k$ be an algebraically closed field of characteristic $p>0$. Let $S$ be an integral normal and proper scheme over $k$ and let $\mX/S$ be a non-constant abelian scheme. Suppose $\mX/S$ has trivial $l$-adic local system, can $\mX/S$ be lifted to $W_2(k)$, the 2-truncated Witt vectors.
\end{question}

By a lift to $W_2(k)$, we mean that there exists a flat proper deformation $\widetilde{S}$ of $S$ to $W_2(k)$ and an abelian scheme $\widetilde{\mX}$ over $\widetilde{S}$, such that $\mX\cong \widetilde{\mX}\times_{\widetilde{\mS}}S$.

Let $C$ be a smooth proper curve over $k$ and $K\coloneq K(C)$ its function field.
We say that an abelian scheme $\mX/C$ has a  small $l$-adic local system if the $l$-adic monodromy representation \[
\rho_l: \etpi(C, \bar{x})\rightarrow \GL_{2g}(\Z_l)
\]has finite image.
A result by Oort (\cite{Oortsubvar} Theorem 2.1) shows that $\mX/C$ has a small $l$-adic local system if and only if  $\mX$ is isogenous to a constant family of abelian varieties up to a finite \'etale base change.



Our main theorem provides a partial answer to Question~\ref{problem 1}

\begin{theorem}\label{Main theorem}
  Let $C$ be a smooth proper curve over an algebraically closed field $k$ of characteristic $p>0$. Let $f:\mX\rightarrow C$ be an abelian scheme with small $l$-adic local system. Then $\mX/C$ cannot be lifted to $W_2(k)$. Moreover, the Hodge bundle $f_*\Omega_{\mX/C}$ is  non-nef.
\end{theorem}

There are many situations in which the small local system condition is satisfied. For example, any family of supersingular abelian varieties over a smooth proper curve satisfies this condition. Using the Kuga–Satake construction, one can also deduce the non-liftability of a family of supersingular K3 surfaces.

\begin{corollary}
  Let $f: \mX\rightarrow C$ be a family of supersingular abelian varieties or degree-$d$-polarized supersingular K3 surfaces over a smooth proper curve. Then $f$ cannot be lifted to $W_2(k)$.
\end{corollary}

\subsubsection{$l$-adic monodromy phenomenon}
Question~\ref{problem 1} is inspired by the behavior of local systems in characteristic zero. Namely, let  $\Sh_K(G, \mX)$ be a Shimura variety associated to a Shimura datum $(G, \mX)$ over the complex numbers. Let $Z\subseteq \Sh_K(G, \mX)$ be an irreducible subvariety. Consider the generic Mumford Tate group $H$ of $Z$ and the associated Shimura subdatum $(H, \mY)\subseteq (G, \mX)$. Let $H_{\text{mon}}^\circ$ be the connected monodromy group of $Z$, defined as the neutral component of the Zariski closure of the image of the monodromy representation $\pi_1(Z, z)\rightarrow \GL(V)$ associated to the Shimura datum $(G, \mX)$. 
    
    Andr\'e 
(see \cite{Andre_fixed_part}) showed that $H_{\text{mon}}^\circ$
is semisimple and that $H^\circ_{\text{mon}}\unlhd H^{\der}$ is a normal subgroup. Considering the connected Shimura datum $(H^\der, \mY')$, we then obtain a decomposition of (connected) Shimura data $$(H^\ad, \mY'^\ad)\cong (H_{\text{mon}}^{\circ\ad}, \mY_1^\ad)\times (H_2^\ad, \mY_2^\ad)$$ for some reductive group $H_2$. Moonen (see \cite{Moonen_linear_I}) showed that for some level structure $K'$, one has $$Z\subseteq Y_1\times \{y_2\},$$ where $Y_1$ is the image of some connected component of $\mY^+_1\subset \mY_1$ and $y_2\in \Sh_{K'}(H_2, \mY_2)$ is a point. In particular, if the monodromy group $H_{\text{mon}}^\circ$ is commutative, then $H^{\circ\ad}_{\text{mon}}$ is trivial and $Z$ is contained in a variety of dimension zero.

From now on, we consider the moduli space of principally polarized abelian varieties $\mA_{g, 1, n}$ with $n\geq 3$. This is a fine moduli space representing the functor$$\mA_{g, 1, n}: S\mapsto \{(X, \lambda, \sigma)\}$$which assigns a locally Noetherian scheme $S$ the set of isomorphism classes of triples with $X/S$ a $g$-dimensional abelian scheme, $\lambda$ a principal polarization on $X$ and $\sigma:(\Z/n\Z)^{2g}\overset{\sim}{\rightarrow}X[n]$ a level-$n$ structure over $X/S$. The scheme $\mA_{g,1, n}$ is defined over $\spec(\Z[\frac{1}{n}])$ and is quasi-projective. In particular, it is a Shimura variety $\Sh_K(\GSp_{2g}, \mH^{\pm})$ of Siegel type. Consequently, it admits no positive-dimensional subvariety with trivial fundamental group.

However, this phenomenon fails in characteristic $p$. As shown in \cite{AST_1981__86__R1_0}, Moret--Bailly first constructed a non-trivial family of principally polarized supersingular abelian surfaces over $\mathbb{P}^1_{\bar{\F}_p}$, now known as the Moret--Bailly family.

\subsubsection{Isogeny leaves}
In this paper, we aim to provide a moduli-wise description for the abelian scheme $\mX/C$ with small $l$-adic local system.

We write $\mA_g$ for $\mA_{g,1,n}, n\geq 3$ when no confusion arises. The scheme
$\mA_{g, \bar{\F}_p}$ has a rich geometric structure admitting several stratifications. We  mainly consider the Newton stratification, which admits an almost product structure  (also called a foliation, see \cite{oort_foliations_2007}). Namely, for each Newton stratum $W_\xi$, there are two kinds of  subvarieties called central leaves $\mC(x)$ and isogeny leaves $I(x)$, which are  transversal to each other. Moreover, for any irreducible component $W\subseteq W_\xi$,there is a finite dominant morphism\[
\mC(x)\times I(x)\twoheadrightarrow W.
\]

A central leaf $\mC(x)$ associated to a polarized abelian variety $x=[A, \lambda]$ is the locus of abelian varieties whose  $p$-divisible groups are isomorphic to $A$ and respecting the polarizations:$$\mC(x)=\{(B,\mu)\in \mA_g(k); (B[p^\infty], \mu[p^\infty])\cong (A[p^\infty], \lambda[p^\infty])\}.$$ The isogeny leaf is, roughly speaking, the maximal integral subvariety in which any family of abelian varieties is isogenous to a constant family (see Section~2 for details).

Consider the $l$-adic local system given by the Tate modules. One can define $G_l(Z)$ to be the algebraic $l$-adic monodromy group of a variety $Z$ with a morphism $\varphi: Z\rightarrow \mA_g$, defined as the Zariski closure of the image of the associated $l$-adic representation. One then considers the neutral component $G_l(Z)^\circ$ of $G_l(Z)$,  called the connected algebraic $l$-adic monodromy group of $Z$.

On the one hand, central leaves have full $l$-adic monodromy group $\GSp_{2g}$ by the Hecke orbit conjecture (now a theorem, see \cite{Chai2005}), while  the isogeny leaves have small $l$-adic monodromy groups by definition. These facts motivate the following slogans: the subvarieties contained in the central leaves exhibit monodromy behavior similar to the characteristic zero case, whereas the subvarieties contained in the isogeny leaves exhibit behavior specific to characteristic $p$.

The following fact is well known to experts. We include a proof in this paper for completeness.

\begin{fact}\label{theorem isogeny leaf}
  Let $Z$ be a smooth projective variety over $k$ and let $\varphi: Z\rightarrow \mA_{g, k}$ be a nontrivial morphism. Then $G^\circ_l(Z)=0$ if and only if $\varphi$ factors through some isogeny leaf $I(x)$.
\end{fact}

Main theorem~\ref{Main theorem} then implies that isogeny leaves are purely characteristic $p$ constructions, and they do not even admit an infinitesimal deformation to $W_2(k)$.

\subsubsection{Positivity of Hodge bundle}

We sketch the proof of the main theorem~\ref{Main theorem}. If $f:\mX\rightarrow C$ admits a $W_2(k)$-lifting, then the theory of characteristic $p$ non-abelian Hodge implies that the graded Higgs bundle $(\mE\oplus\mE^\vee, \theta)$ of the first de Rham cohomology with Gauss--Manin connection $(\mbH^1_{dR}(\mX/C), \nabla_{GM})$ is Higgs semistable. This can be found in Ogus--Vologodsky (see \cite{ogus_nonabelian_2007}) or the recent theory of Higgs--de Rham flow developed by Lan--Sheng--Zuo (see \cite{lan_semistable_2017}). Here, $\mE=f_*\Omega_{\mX/C}^1$ is the Hodge module and $\mE^\vee=\Lie(\mX/C)$ is the Lie algebras (see Appendix~\ref{appendix} for background on Higgs semistability).
In particular, the $W_2(k)$-liftability implies the positivity of the Hodge bundle, i.e.,  $\mu_{\min}(f_*\Omega^1_{\mX/C})\geq 0$.

On the other hand, we recall a  result of Oort (\cite{Oortsubvar}, Theorem 2.1): if $\mX/C$ is an abelian scheme with trivial $l$-adic monodromy group, then up to some \'etale cover $C'\rightarrow C$, the $K(C)/k$-trace morphism \[\tau_{K(C)/k}:\Tr_{K(C)/k}\mX_{K(C)}\times_kC\rightarrow \mX\] is an isogeny. This statement is implicit in the proof of Theorem~\ref{theorem isogeny leaf}.

We show that if $\mu_{\max}(\Lie(\mX/C))\leq 0$, then $\tau_{K(C)/k}$ is in fact an isomorphism, contradicting the assumption that $\mX/C$ is  non-constant. Thus one must have $\mu_{max}(\Lie(\mX/C))>0$. Equivalently, $\mu_{\min}(f_*\Omega^1_{\mX/C})<0$. This contradicts the $W_2(k)$-liftability of $\mX/C$ and furthermore implies that $f_*\Omega^1_{\mX/C}$ is non-nef, as asserted in the main theorem~\ref{Main theorem}.

More generally, we show the nontrivial and non-constant trace already implies the non-nefness.

\begin{theorem}\label{theorem non nef}
 Suppose that the universal family $f:\mX\rightarrow C$ has maximal variation, i.e., it does not admit a constant abelian subscheme. If $\Tr_{K(C)/k}\mX_{K(C)}\neq 0$, then $f_*\Omega_{\mX/C}^1$ is not nef.
\end{theorem}

The non-nefness of the  Hodge bundle of the abelian scheme $\mX/C$ is another  phenomenon specific to characteristic $p$. As in characteristic zero, it was first shown by  Griffiths in \cite{Griffiths1970} (see also \cite{Bost+2004+371+418} Corollary 2.7) that $\mE$ is nef. On the other hand, the conditions of Theorem~\ref{theorem non nef} can never occur in characteristic zero, since the trace morphism is purely inseparable. Thus, in characteristic zero, $\Tr_{K(C)/k}\mX_{K(C)}\neq 0$ if and only if $\mX$ has a constant abelian subscheme.

We therefore conjecture that if the Hodge bundle is non-nef, then it should have a factor determined by some isogeny leaf. In other words, 
the non-nefness of Hodge bundle and the non-triviality of trace should be pure characteristic $p$ phenomena and should determine each other. Hence we propose the following conjecture.

\begin{conjecture}\label{conj nonnef}
  Let $\mX/C$ be an abelian scheme and suppose it does not admit an isotrivial abelian subscheme. 
  If $\mu_{\min}(f_*\Omega_{\mX/C}^1)<0$, then for some finite \'etale cover $C'\twoheadrightarrow C$ the $K(C')/k$-trace morphism $$\tau_{K(C')/k}: (\Tr_{K(C')/k}\mX_{K(C')})_{K(C')}\rightarrow X_{K(C')}$$ is nonzero and cannot be descended to $k$.
\end{conjecture}

It should be noted that Conjecture~\ref{conj nonnef} concerns the converse direction of a result of Yuan:

\begin{theorem}[\cite{Yuan_2021}, Theorem 3.9]\label{Yuan}
  Let $C$ be a projective and smooth curve over a finite field $k$ of characteristic $p>0$, and $K=K(C)$. Let $A$ be an abelian variety over $K$ with $p$-rank 0 and trivial $K/k$-trace. Suppose $A$ has semistable reduction everywhere over $C$. Then the Hodge bundle of $A$ is not ample over $C$.
\end{theorem}

We also establish an Arakelov-type inequality for an abelian scheme $\mX/C$ under the assumption of a $W_2(k)$-lifting.

\begin{theorem}[Arakelov inequality]
  Let $f: \mX\rightarrow C$ be a family of principally polarized  abelian varieties of relative dimension $g$ over a proper smooth $k$-curve. Assume that $f$ can be lifted to $W_2(k)$. Then \[\deg f_*\Omega_{\mX/C}^1\leq g(g(C)-1).\]
\end{theorem}

\subsubsection{Constant subgroup of $\text{BT}_1$}In the proof of Theorem~\ref{theorem non nef}, we construct a constant group subscheme $H_0\times C\subseteq \mX[F]$ arising  from an isogeny from a  constant abelian scheme. We eventually show that such a constant group subscheme $H_0\times C\subseteq \mX[F]$ comes from a constant abelian subscheme of $\mX$.
We expect that this holds more generally: the constancy of $\mX[F]$ should imply the isotriviality of $\mX$. We therefore  raise the following question:

\begin{question}
  If $H\times_kC$ is a constant group subscheme of $\mX[F]$, then does $\mX$ admit an isotrivial abelian subscheme $B_0$ with $B_0[F]\cong H$?
\end{question}

\begin{Organization}
In Section~2, we study the isogeny leaves using $l$-adic local system and the theory of traces. We show the characteristic zero non-liftability of quasi-projective or proper subvarieties of isogeny leaves.
In Section~3, we study the positivity properties of Hodge bundles of abelian schemes over curves and prove an Arakelov-type inequality assuming $W_2(k)$-liftability.
In Section~4, we prove the main theorem by proving the non-nefness of the Hodge bundle for a curve inside an isogeny leaf.
In Section~5, we mention some interesting problems that we plan
to investigate in future work.

\end{Organization}

\begin{acknowledgment}
The author thanks his advisor Ananth Shankar for introducing this problem and the theory of Higgs--de Rham flows of Lan--Sheng--Zuo. The author also thanks to Xinyi Yuan,  Jakub Witaszek, Ruofan Jiang and Shun Yin for helpful discussions. The first four chapters are independent results motivated by \cite{Oortsubvar} and \cite{ogus_nonabelian_2007}, together with an original intuition about the $l$-adic local system. The author later noticed the work of Yuan \cite{Yuan_2021} and shifted the emphasis to the positivity of Hodge bundles, which also motivates the writing of Chapter~5.

\end{acknowledgment}

\section{Subvarieties of isogeny leaves and $l$-adic monodromy}

\subsection{Preliminaries on isogeny Leaves}

Let $k$ be an algebraically closed field of characteristic $p>0$. Let $\mA_{g, d,  n, k}$ denote the (fine) moduli space of principally polarized abelian varieties of dimension $g$ with degree-$d$ polarization and a level-$n$ structure over $k$, where $n\geq 3$ and $(n, p)=1$. When no confusion arises, we write $\mA_g$ for $\mA_{g, 1, n, k}$.
The isogeny leaves can be considered as mod  $p$ reduction of the Rapoport-Zink spaces (see \cite{rapport-zink}). Alternatively,  Oort (see \cite{oort_foliations_2007}) defined the isogeny leaves as the maximal integral $H_\alpha$-scheme.

\begin{definition}[$H_\alpha$-scheme]
  Let $S$ be a $k$-scheme and let $(\mA, \lambda)\rightarrow S$ be a polarized abelian scheme over $S$. A reduced scheme $I\subseteq S$ is called an $H_\alpha$-scheme if there exists a polarized abelian variety $(B, \mu)$ over $k$, a scheme $T$ of finite type over $k$, and a surjective morphism $T\twoheadrightarrow I$ together with an isogeny $$\varphi: (B, \mu)\times T\rightarrow (\mA, \lambda)\times_IT$$ such that every geometric fibre of $\varphi$ is of local-local type. The pair $(T\rightarrow I, \varphi)$
is called a chart of the  $H_\alpha$ scheme. 

\end{definition}

\begin{remark}
Note that  for an $H_\alpha$-scheme $S\subset\mA_{g, 1, n} $ with isogeny $\varphi: (B, \mu)\times T\rightarrow (\mA, \lambda)\times_IT$ as in the definition, we do not require $\mu$ to be principal. Indeed, $(B, \mu)\in \mA_{g, d, n}(k)$, where $d=\deg(\varphi)$. Here, we adopt the convention that a polarization $\lambda=\varphi_L$ for some ample line bundle $L$ on $A$ has degree $d$ if $\chi(L)=d$. So, $\deg(\varphi_L)=d^2$.
\end{remark}

\begin{lemma}\label{local-local gpsch}
  A finite group scheme $G$ over a perfect field $k$ of characteristic $p>0$ is of local-local type if and only if $G$ is a successive extension of $\alpha_p$ over $k$.
\end{lemma}

\begin{proof}
  Let $p^n$ be the rank of $G$. By Dieudonn\'e theory, it suffices to show a finite-length Dieudonn\'e module $(M, F, V)$ over $W(k)$ has nilpotent $F, V$ if and only if $M$ has a Dieudonn\'e submodule $(k, 0, 0)\subseteq (M, F, V)$ and $$(M, F, V)/(k, 0, 0)$$ also has nilpotent $F, V$. 
  
  One direction is clear. Conversely, one can consider the Dieudonn\'e submodule $(\Ker(F), 0, V)$ which is nontrivial as $F$ is nilpotent. Similarly, $(\Ker(F)\cap\Ker(V), 0, 0)$ is a nontrivial Dieudonn\'e submodule of $(M, F, V)$, with $\Ker(F)\cap \Ker(V)\cong k^{\oplus m}$ for some $m$. This gives the desired submodule $(k ,0 ,0)\subseteq (M, F, V)$. The quotient has nilpotent $F, V$ obviously.

\end{proof}

More generally, a finite flat group scheme $G$ over a $k$-scheme $S$ is of local-local type if and only if $G$ is locally isomorphic to a successive extension of $\alpha_p$ over $S$. This explains the name of $H_\alpha$-scheme. The isogeny leaf will be considered as the irreducible component of maximal $H_\alpha$-scheme.

\begin{theorem}[Oort, \cite{oort_foliations_2007}]\label{maximal H scheme}

  Let $[(A, \lambda)]=x\in \mA_g(k)$ be a point corresponding to a principally polarized abelian variety over $k$. Then there exists a reduced, closed subscheme $\mI(x)\subset \mA_g$ which is a $H_\alpha$-scheme such that every irreducible component of $\mI(x)$ contains $x$ and $\mI(x)$ is the union of all irreducible $H_\alpha$-schemes containing $x$. 
\end{theorem}

\begin{definition}[Isogeny leaf]
  Let $x=[(A, \lambda)]\in\mA_g(k)$ and let $\mI(x)\subset \mA_g$ as in the above theorem. An irreducible component of $\mI(x)$ with the induced reduced subscheme structure is called an isogeny leaf through $x$ and  denoted by $I(x)$. 
\end{definition}
\begin{remark}
  Alternatively, in \cite{ChaiOort_ModuliAV}, Chai and Oort defined the isogeny leaves by using quasi-$\alpha$-isogenies. Let $(A, \lambda), (B, \mu)$ be two principally polarized abelian varieties over $k$. Then a quasi-$\alpha$-isogeny $\psi: (A, \lambda)\dashrightarrow (B, \mu)$ is a diagram\[(A, \lambda)\overset{\phi}{\longleftarrow}(M, \nu)\overset{\varphi}{\longrightarrow}(B, \mu),\]
where $(M, \nu)$ is a polarized abelian variety such that $\phi^*\lambda=\nu=\varphi^*\mu$ and $\phi, \varphi$ are isogenies of local-local type. The set consisting of all principally polarized $(B, \mu)$ that are quasi-$\alpha$-isogenous to $x=[(A, \lambda)]\in \mA_{g, 1, n, k}$ is called a Hecke-$\alpha$-orbit of $x$, denoted by $H_\alpha(x)$. An isogeny leaf $I(x)$ is defined to be the union of
all irreducible components of Hecke-$\alpha$-orbit of $x$. 

One should note that the two definitions of $H_\alpha$-schemes are equivalent: let $\phi: (M, \nu)\rightarrow (A,\lambda)$ be a local-local isogeny with $\phi^*\lambda=\nu$. Since $\lambda$ is principal, $\nu: M\rightarrow M^\vee$  is an isogeny of local-local type. 
It is known that $\Ker(\nu)$ contains a Lagrangian subgroup scheme $H$ such that $\nu$ descends to a principal polarization on $B\coloneq M/H$, i.e., we have a principal polarization $\mu: B\rightarrow B^\vee$ with an isogeny of local-local type preserving polarization $\varphi: (M, \nu)\rightarrow (B, \mu)$.

\end{remark}

In other words, an isogeny leaf is a maximal integral $H_\alpha$-subscheme of $\mA_g$. 
Unlike the central leaves which are quasi-affine, the isogeny leaves are proper and hence would admit proper subvarieties. 

\begin{theorem}[Oort, \cite{oort_foliations_2007}]
  An irreducible, maximal $H_\alpha$-scheme $I\subset \mA_g$ is closed in $\mA_g$ and proper over $k$.
\end{theorem}

By construction of the isogeny leaf, we obtain the following characterization of subschemes contained in the isogeny leaf.

\begin{theorem}\label{subvar in isogeny leaf}
  Let $f: Z\rightarrow \mA_g$ be a morphism from an integral scheme $Z$ over $k$. Then $f$ factors through some isogeny leaf $I(x)$ if and only if possibly adding some level structure, there exists a principally polarized abelian variety $(B, \mu)$ over $k$ and an isogeny $$\phi: (B, \mu)\times_{\spec(k)} Z\rightarrow \mA$$ over $Z$ such that every geometric fibre of $\phi$ is of local-local type, where $\mA\rightarrow Z$ is the family of principally polarized abelian varieties induced by $f$.

\end{theorem}

\begin{proof}
  
Let $f: Z\rightarrow I(x)$ be a given map from $Z$ to the isogeny leaf $I(x)$. By \cite{oort_foliations_2007} Corollary 4.7, one can take the chart $T\twoheadrightarrow I(x)$ to be generically finite and proper. Hence, after considering an appropriate level structure, we conclude we have an isogeny $(\mA, \lambda)\rightarrow (B, \mu)\times I(x)$ expressing the $H_\alpha$-scheme structure of $I(x)$. Pulling back this isogeny along $f: Z\rightarrow I(x)$, we get the desired isogeny over $Z$ with respect to some level structure.

On the other hand, suppose that there exists a polarized abelian variety $(B, \mu)$ over $k$ and an isogeny $\phi: (B, \mu)\times_{\spec(k)} Z\rightarrow \mA$ over $Z$ such that every geometric fibre of $\phi$ is of local-local type. Without loss of generality, we assume $(B, \mu) $ is a geometric fibre of $\mA$ over $x\in Z(k)$. By the definition of $H_\alpha$-scheme, we know that the image of $f: Z\rightarrow \mA_g$ is contained in some $H_\alpha$-scheme. Since the isogeny leaf $I(x)$ is the maximal integral $H_\alpha$-scheme containing $x$ by Theorem~\ref{maximal H scheme}, we know that $f$ factors through some isogeny leaf.

\end{proof}

Finally, we explain the foliation structure provided by central leaves and isogeny leaves inside Newton strata.
Let $\eta$ be a Newton polygon of some abelian variety of dimension $g$. Recall that the Newton stratum 
$W_\eta$ is equidimensional of dimension $\triangle(\eta)$, which only depends on $\eta$ (see \cite{chai_moduli_2007}). And if $[(A, \lambda)]=x\in \mA_g(k)$, the central leaf $$\mC(x)=\{(B,\mu)\in \mA_g(k); (B[p^\infty], \mu[p^\infty])\cong (A[p^\infty], \lambda[p^\infty])\}$$ is equidimensional of dimension $c(\eta)$, which depends only on the Newton polygon of $A$ (see \cite{oort_foliations_2007}). In particular, when $\xi$ is the supersingular Newton polygon, we have $\triangle(\xi)=\left \lfloor \frac{g^2}{4}\right \rfloor$ and $c(\xi)=0$. 

\begin{theorem}[Oort, \cite{oort_foliations_2007}]\label{almostproduct}
    Let $n\geq 1$ and let $\eta$ be a Newton polygon of some abelian variety of dimension $g$. Let $W\subset W_\eta^\circ$ be an irreducible component of the open Newton stratum $W_\eta^\circ$. Then there exist integral schemes $T$ and $J$ of finite type over $k$ and a finite surjective $k$-morphism
\[\Phi: T\times J\twoheadrightarrow W\subset \mA_{g, n},\]
such that\[\forall u\in J(k), \Phi(T\times\{u\})\text{ is a central leaf in }W,\]
and every central leaf in $W$ can be arised in this way.
Similarly, \[\forall v\in T(k), \Phi(\{v\}\times J)\text{ is an isogeny leaf in }W,\]
and every isogeny leaf in $W$ can be arised in this way.

\end{theorem}

In particular, Theorem~\ref{almostproduct} implies that every irreducible component of the supersingular locus $W_\xi$ is an isogeny leaf by the dimension calculation before, since in this case $T$  must be a point. \\

\subsection{Preliminaries on trace of abelian varieties}

Let $K/k$ be a primary extension of fields, i.e., the algebraic closure of $k$ in $K$ is
purely inseparable over $k$. If moreover, $k$ is algebraically closed in $K$ and $K/k$ is separable, then we also call such $K/k$ a regular extension. 
For example, if $k$ is perfect and $K$ is the function field of  a smooth and geometrically connected
$k$-scheme, $K/k$ is a regular extension.

\begin{definition}[Chow's $K/k$-trace]
  Let $K/k$ be a primary extension and let $A$ be an abelian variety over $K$. The $K/k$-trace is a final object in the category of pairs $(B, \phi)$, where $B$ is a $k$-abelian variety and $\phi: B_K\rightarrow A$ is a $K$-homomorphism of abelian varieties. 

\end{definition}

We refer the reader to read Brain Conrad's article \cite{Conrad2006ChowsKA} for more details about the $K/k$-trace. We recall the following properties that we will use later.

\begin{theorem}[\cite{Conrad2006ChowsKA}, Section 6]\label{trace}
Let $K/k$ be a primary extension and let $A$ be an abelian variety over $K$.

  (1).  The $K/k$-trace \[\tau=\tau_{A, K/k}:(\Tr_{K/k}A)_K\rightarrow A\]always exists. Moreover, $\Ker(\tau)$ is a finite flat group scheme over $K$ with connected Cartier dual.

  (2). If $K/k$ is furthermore a regular extension, then $\Ker(\tau)$ is connected.

  (3). There exists a unique abelian subvariety $A'\subseteq A$ such that $\Tr_{K/k}(A/A')=0$ and $\tau_{A', K/k}:\Tr_{K/k}(A')_K\rightarrow A'$ is an isogeny.

  (4). Let $k'/k$ be an arbitrary extension. Suppose that either $k'/k$ is separable or $K/k$ is regular. The the base change $(\Tr_{K/k}A)\otimes_kk', \tau\otimes_{K}Kk' )$ gives the $Kk'/k'$-trace of $A_{Kk'}$.
  
\end{theorem}

In particular, we see that if $K/k$ is a regular extension, the $K/k$-trace $\tau: (\Tr_{K/k}A)_K\rightarrow A$ gives a $K$-isogeny from $(\Tr_{K/k}A)_K$ onto its image. It is not always the case that $\tau$ is an isogeny itself.

On the other hand, $\Tr_{K/k}A$ is the natural object for studying the isotriviality of an abelian scheme $\phi: \mA\rightarrow C$ over a smooth proper $k$-curve $C$. Let $k$ be a perfect field and let $K=K(C)$ be the function field, so that $K/k$ is a regular extension. Let $A=\mA_K$ be the generic fibre of $\mA$.

By the theory of N\'eron models, the trace map $\tau: (\Tr_{K/k}A)_K\rightarrow A$ gives a homomorphism of abelian schemes $$\widetilde{\tau}:  \Tr_{K/k}A\times_{\spec(k)}C\rightarrow \mA$$ over $C$. By the universal property of the trace, any homomorphism from a constant abelian scheme $B\times C\rightarrow \mA$  factors through $\widetilde{\tau}$. So, the existence of a nontrivial homomorphism from a constant abelian scheme to $\mA$ forces the $K/k$-trace to be nontrivial.

Conversely, if $\text{char}(k)=0$,  $\Ker(\tau)$ has to be trivial by Theorem~\ref{trace}. In particular, if $B\times C\hookrightarrow \mA$ is an injective homomorphism from a constant abelian scheme, then $B$ is an abelian subvariety of $\Tr_{K/k}A$ by the universal property of the trace.

Similarly, if $\text{char}(k)=p$,  $\Ker(\tau)$ has to be trivial if it can be descended to $k$.

\subsection{$l$-adic monodromy behavior of subvarieties in isogeny leaves}
Let $Z$ be a geometrically connected integral scheme over $k$ and  let $\varphi: Z\rightarrow \mA_g$ be a morphism. It gives a family of principally polarized abelian varieties over $Z$, say $\mA\rightarrow Z$. Let $\bar{z}\in Z(k)$ be a geometric point and consider the $l$-adic local system given by the $l$-adic rational Tate modules $$\rho_l: \etpi(Z, \bar{z})\rightarrow \GL(T_l(\mA_{\bar{z}})).$$ One can then define the algebraic $l$-adic monodromy group on $Z$ as  the Zariski closure of the image of $\rho_l$ in $\GL(T_l(\mA_{\bar{z}}))\subset \GL(V_l(\mA_{\bar{z}}))\cong \GL_{2g}(\Q_l)$. We will denote this group by $G_l(Z)$. 

If $Z$ is furthermore normal, there is a natural surjection\[\etpi(\eta, \bar{\eta})\twoheadrightarrow \etpi(Z, \bar{\eta}),\]from  the \'etale fundamental group of the generic point $\eta$ of $Z$. So,  $G_l(Z)$ coincides with the $l$-adic monodromy group of the Galois group $\Gal(K(Z)^s/K(Z))$ acting on the $l$-adic Tate module $V_l(\mA_{\bar{\eta}})$ of the generic fibre, where $K(Z)$ is the function field of $Z$.
By Tate's $l$-adic isogeny theorem, $\rho_l$ is a semisimple representation restricted to the generic fibre and hence $G_l(Z)$ is a reductive group over $\Q_l$.

We write $G_l(Z)^\circ$ for the neutral component of $G_l(Z)$, which is called the connected algebraic $l$-adic monodromy group of $Z$.


\begin{remark}\label{trivial l-adic connected algebraic monodromy}
  In the proof of Theorem~\ref{subvar in isogeny leaf}, one sees that we can add a level structure so that $(\mA, \lambda)\rightarrow (B,\mu)\times I(x)$ is a purely inseparable isogeny expressing the $H_\alpha$-scheme structure of $I(x)$. In particular, this implies that $\rho_l:\etpi(I(x), \bar{x})\rightarrow \GL(T_l(\mA_{\bar{x}}))$ has finite image, and hence $G_l(I(x))^\circ$ is trivial. 

  Note also that $G_l(I(x))^\circ$ is not reductive generally, as $I(x)$ is not necessarily normal.
\end{remark}

Our first observation is the following characterization of the subvarieties of isogeny leaves in terms of $l$-adic monodromy groups.

\begin{theorem}\label{trivial l-adic monodromy}
  Let $C$ be a smooth proper irreducible curve over $k$ and let $\varphi: C\rightarrow \mA_g$ be a nontrivial map. Then $\varphi$ factors through some isogeny leaf $I(x)$ if and only if the connected algebraic $l$-adic monodromy group $G_l(C)^\circ$ is trivial. 
\end{theorem}

The theorem is inspired by the following result of Oort.

\begin{theorem}[\cite{Oortsubvar} Theorem 2.1]\label{oort}
Let $C$ be a smooth proper curve over a field $k$. Let $\mA/C$ be an abelian scheme and let $K=K(C)$ be the function field of $C$. Let $A\coloneq \mA_K$ be the generic fibre of $\mA$. If the image of the $l$-adic monodromy representation\[\rho_l:\Gal(K^s/K)\rightarrow \GL_{2g}(\Z_l)\]has commutative image. Then there exists an unramified covering between smooth proper curves $C'\twoheadrightarrow C$ such that the $K'/k$-trace map\[\tau_{K'/k, A_{K'}}: (\Tr_{K'/k}A)_{K'}\rightarrow A_{K'}\]is an isogeny, where $K'=K(C')$. In particular, this shows that the abelian scheme $\mA\times_CC'$ is isogenous to a constant abelian scheme $B\times_{\spec(k)}C'$ via a purely inseparable isogeny, where $B=\Tr_{K'/k}A_{K'}$.

\end{theorem}

\begin{proof}[Sketch of proof: ]
    We only prove the case when $k=\bar{k}$ and $\rho_l$ has trivial image.
    Consider the $K/k$-trace morphism\[
    \tau:(\Tr_{K/k}A)_K\rightarrow A,
    \]which gives an exact sequence\[
    (\Tr_{K/k}A)_K\rightarrow A\rightarrow B\rightarrow 0\tag{$*$}
    \]for some abelian variety $B$ over $K$. 

    Since $\tau$ has a finite connected kernel, $(*)$ induces a short exact sequence of $l$-adic Tate modules\[
    0\rightarrow T_l(\Tr_{K/k}A)\rightarrow T_lA\rightarrow T_lB\rightarrow 0.\tag{$**$}
    \]
By the assumption on $\rho_l$, $\Gal(K^s/K)$ acts trivially on all these three modules. In particular, one has $B[l^\infty](K^s)\cong B[l^\infty](K)$. Let $\Tr_{K/k}B$ be the $K/k$-trace of $B$. Then a theorem of Lang-N\'eron (see \cite{Conrad2006ChowsKA} Theorem 7.1) says that $B(K)/\Tr_{K/k}B(k)$ is a finitely generated abelian  group. So, the $l$-power torsions $B[l^\infty](K)/\Tr_{K/k}B[l^\infty](k)$ is a finite group. 
The identifications $B[l^\infty](K^s)\cong B[l^\infty](K)$ then implies the embedding $T_l(\Tr_{K/k}B)\hookrightarrow T_lB$ induced by the $K/k$-trace morphism has finite cokernel. In particular, the $K/k$-trace morphism of $B$
is an isogeny.

Consider the $K$-isogeny $B\oplus (\Tr_{K/k}A)_K\rightarrow A$ induced by $(*)$. It then induces a $K$-isogeny \[
(\Tr_{K/k}B\oplus\Tr_{K/k}A)_K\rightarrow A.
\]The universal property of trace then implies $\tau: (\Tr_{K/k}A)_K\rightarrow A$ is an isogeny.

\end{proof}

\begin{proof}[Proof of \ref{trivial l-adic monodromy}]
  Suppose first that $\varphi$ factors through an isogeny leaf, i.e., $\varphi: C\rightarrow I(x)$. By Remark~\ref{trivial l-adic connected algebraic monodromy}, $I(x)$ has trivial connected algebraic $l$-adic monodromy group. So, the same holds for $C$ as $\varphi$ factors through $I(x)$.

  Conversely, suppose $G_l(C)^\circ$ is trivial. This implies that $\rho_l:\etpi(C, \bar{x})\rightarrow \GL_{2g}(\Z_l)$ has finite image. So, there exists a finite \'etale cover $C'\rightarrow C$ such that $C'$ has trivial $l$-adic monodromy image. In particular, the $l$-adic monodromy representation associated to the generic point $$\rho_l: \Gal(K(C')^s/K(C'))\rightarrow \GL_{2g}(\Z_l)$$ has commutative image. By Theorem~\ref{oort}, after pulling back to an unramified covering between smooth proper $k$-curves $C''\rightarrow C'$, the induced abelian scheme $$(\mA\times_CC')\times_{C'}C''\cong \mA\times_CC''$$ is isogenous to a constant abelian scheme $B\times_{\spec(k)}C''$ via a local-local type isogeny, with $B=\Tr_{K(C'')/k}A_{K(C'')}$.

  Let $L=K(C'')$ and $\phi: B\otimes_kL\rightarrow A\otimes_KL$ the $L/k$-trace morphism (isogeny). The principal polarization $\mu$ on $A_L$ induces an $L$-polarization $\mu'=\phi^*\mu\coloneq \phi^\vee\circ\mu\circ\phi$ of $B_L$.

Since
the group scheme $Hom_k(B, B^\vee)$ is \'etale over $k$, one has \[\Hom_k(B, B^\vee)\cong \Hom_L(B_L, B^\vee_L),\]
and hence $\mu'$ descends to a $k$-isogeny of B$$\mu'_0: B\rightarrow B^\vee.$$ It is known that if a $k$-isogeny $\varphi$ becomes a polarization after base change to some field extension $K/k$, then so is $\varphi$ over $k$. In particular, $\mu'_0$ is a polarization of $B$ over $k$.

  Finally, the theory of N\'eron model again implies that  we have found an isogeny $\phi: (B, \mu_0')\times_{\spec(k)}C''\rightarrow (\mA, \mu)\times_CC''$ of local-local type for some unramified cover $C''\twoheadrightarrow C$.
  This implies that the scheme theoretic image of $C$ in $\mA_g$ is an $H_\alpha$-scheme by definition. So, $$\varphi: C\rightarrow \mA_g$$ factors through some isogeny leaf by Theorem~\ref{subvar in isogeny leaf}.

\end{proof}

\begin{corollary}\label{cor1}
  Let $\varphi: C\rightarrow I(x)$ be a morphism from a smooth proper irreducible curve $C$ to some isogeny leaf $I(x)$. Then there exists an \'etale $k$-covering $D\twoheadrightarrow C$ such that there exists a local-local type isogeny $\phi: B\times_{\spec(k)}D\rightarrow \mA\times_CD$, where $\mA\rightarrow C$ is the family of abelian varieties induced by $\varphi$.
\end{corollary}

\begin{proof}
  Indeed,  an unramified covering between irreducible smooth proper curves is flat (see \cite{hartshorne_algebraic_1977}, Chapter III, Proposition 9.7), hence \'etale.
\end{proof}




Theorem~\ref{trivial l-adic monodromy} can be generalized to higher dimensional varieties as follows.

\begin{theorem}\label{trivial l-adic monodromy for general projective variety}
    Let $\varphi: Z\rightarrow \mA_g$ be a morphism from a smooth integral projective variety $Z$ over $k$. Then  $G_l(Z)^\circ$  is trivial if and only if $\varphi$ factors through some isogeny leaf $I(x)$.
\end{theorem}

\begin{proof}
   One direction is similar as before.

 Conversely,  take $x, y\in Z(k)$ to be arbitrary geometric points. Since $Z$ is projective, Bertini's theorem shows that there exists a finite number of smooth proper curves $C_i$ connecting $x$ and $y$. Every such curve $C_i$ has trivial connected algebraic $l$-adic monodromy group and hence is a  $H_\alpha$-scheme as in the proof of Theorem~\ref{trivial l-adic monodromy}. Suppose $C_i$ connects $y_i$ and $y_{i+1}$. Since $I(y_{i+1})$ is by definition the maximal integral $H_\alpha$-scheme containing $y_{i+1}$, we have $y_i\in C\subset I(y_{i+1})$. Again, since $I(y_i)$ is the maximal integral $H_\alpha$-scheme containing $y_i$, we have $I(y_i)\subset I(y_{i+1})$. As the choice of index $i$ is arbitrary, this implies that $I(y_i)=I(y_{i+1})$. Repeating this argument, we get $I(x)=I(y)$. In particular, all these curves $C_i$ lie in a single isogeny leaf $I(x)$.
 
 Since $x$ and $y$ are arbitrary, this shows that $\varphi: Z\rightarrow \mA_g$ factors through some isogeny leaf.
\end{proof}



 


\begin{remark}\label{remark 1}
Consider the situation in Corollary~\ref{cor1}. Suppose that $f: \mA\rightarrow C$ lifts to $W_2(k)$, say $f: \tilde{\mA}\rightarrow \tilde{C}$. Then this gives $\mA\times_CD\rightarrow D$ a lift. Indeed, since $D\rightarrow C$ is  \'etale,  the cotangent complex $L_{D/C}\cong 0$. So, the obstruction $$ob(D\rightarrow C)\in \Ext^2(L_{D/C}, O_D\otimes_k I)\cong 0$$ is trivial. As a result, such $D\rightarrow C$ admits a lifting $\tilde{D}\rightarrow \tilde{C}$ over $W_2(k)$. The base change $\tilde{\mA}\times_{\tilde{C}}\tilde{D}$ then gives a lift of $\mA\times_CD\rightarrow D$. 

So, without loss of generality, we might sometimes assume that there exists an isogeny $\phi: B\times_{\spec(k)}C\rightarrow \mA$ of local-local type over $C$ without taking an \'etale cover or adding a level structure as in Theorem~\ref{subvar in isogeny leaf} if $\varphi:C\rightarrow \mA_g$ factors through some isogeny leaf.

On the other hand,  one has $$2g(D)-2=d\cdot (2g(C)-2)$$by the Hurwitz formula,, where $d$ is the degree of the \'etale cover $D\rightarrow C$. In particular, if $g(C)=0$ or 1, it is always the case that the abelian scheme $\mA/C$ is isogenous to a constant abelian scheme. 

\end{remark}

\begin{remark}\label{remark 2}
    

As explained in the introduction, if $Z\subseteq\Sh_K(G, \mX)$ is a positive dimensional subvariety of a complex Shimura variety, then the connected monodromy group $H_{\text{mon}}^\circ(Z)$ cannot be commutative. 

In general, there cannot be irreducible positive-dimensional subvariety contained in Shimura variety over characteristic zero with commutative monodromy group. Moving to $l$-adic monodromy, this implies only zero-dimension subvariety would have trivial $l$-adic monodromy group by the compatibility (in the \'etale topology) $\varprojlim_n (R^1f_*\Z/l^n\Z)\cong (R^1f^{\text{an}}_*\Z)\otimes\Z_l$. 

However, Theorem~\ref{trivial l-adic monodromy} shows that in the special fibre $\mA_{g, \bar{\F}_p}$ of the canonical model $\mA_g$ of the Shimura variety of Siegel type $\Sh_K(\GSp_{2g}, \mX)$, positive-dimensional subvarieties with trivial connected
l-adic monodromy can occur, and the only possibility is that they lie in an isogeny leaf. 

On the other hand, one expects that if $Z$ is a smooth variety contained in a good locus  (e.g. ordinary locus) of $\mA_{g, \bar{\F}_p}$, then the connected algebraic $l$-adic monodromy group $G_l^\circ(Z)$ plays the same role as the generic Mumford-Tate group in the characteristic zero case. This is usually called the Mumford-Tate Conjecture for $\mA_{g, \bar{\F}_p}$, see, for example, \cite{jiang_characteristic_2024} Conjecture 1.1.

\end{remark}

The $l$-adic monodromy criterion together with the previous discuss suggests the subvarieties of isogeny leaves cannot be lifted to characteristic zero.

\begin{theorem}\label{cannot be lifted to char 0}
Let $Z$ be a smooth  proper irreducible variety over $k$ and let $\varphi: Z\rightarrow I(x)$ be a nontrivial map, where $I(x)$ is some isogeny leaf. Then $\varphi: Z\rightarrow I(x)$ cannot be lifted to $W(k)$.
\end{theorem}
\begin{proof}

  Suppose that the $\varphi: Z\rightarrow \mA_{g, \bar{\F}_p}$ lifts to an abelian scheme  $\tilde{\varphi}: \widetilde{Z}\rightarrow \mA_{g, W(k)}$, with $\title{Z}$ smooth proper over $W(k)$. It gives a specialization map $sp: \etpi(\widetilde{Z}_K, \bar{x})\rightarrow\etpi(Z, \bar{y})$, where $K$ is the quotient field of $W(k)$. Namely, $sp$ is given by the composition of maps\[\etpi(\widetilde{Z}_K, \bar{x})\rightarrow\etpi(Z, \bar{x})\cong \etpi(Z_k, \bar{y}),\]where the isomorphism is because $W(k)$ is henselian.  It is known that $sp$ is a surjective (see  \cite[\href{https://stacks.math.columbia.edu/tag/0BUQ}{Tag 0BUQ}]{stacks-project}) if $Z$ is smooth.  
  

  However, the connected algebraic $l$-adic monodromy group of $Z$ is trivial by Theorem~\ref{trivial l-adic monodromy}, but the connected algebraic $l$-adic monodromy group of $\widetilde{Z}_K$ is absolutely nontrivial by Remark~\ref{remark 2}. This gives a contradiction.

\end{proof}

\begin{remark}
    It is worth noting that smoothness is essential in  Theorem~\ref{cannot be lifted to char 0}. For example, the work of Terakado--Xue--Yu \cite{chia2023supersingularlocusshimuravarieties} shows that there exists a Shimura curve $\mM_K$ of PEL type whose reduction modulo some prime $p$ of the canonical integral model is totally supersingular, and with ordinary singularities exactly at superspecial points. Moreover, \cite{chia2023supersingularlocusshimuravarieties} Theorem 1.3 shows that the superspecial locus is non-empty. As Ruofan Jiang commented, one may find a proper Shimura curve whose reduction is totally supersingular, but we don't expect it to be smooth anymore by Theorem~\ref{cannot be lifted to char 0}.

    This remark is due to Ruofan Jiang for first pointing out the possible counter-example of Problem~\ref{problem 1}.

\end{remark}

\begin{remark}
        One may also ask for the properness in the condition of Theorem~\ref{cannot be lifted to char 0}: can a general smooth irreducible variety contained in one isogeny leaf be lifted to characteristic zero? Let $C$ be a smooth proper curve over $k$ and let $S\subset C$ be a finite subset of geometric points. Take $U\coloneq C\setminus S$ be a non-proper curve.   Let $\varphi: U\rightarrow I(x)$ be a nontrivial map from $U$ to some isogeny leaf and let $\tilde{\varphi}:\widetilde{U}\rightarrow\mA_{g, W(k)}$ be a smooth lift of $\varphi$ over $W(k)$. In \cite{grothendieck_revetements_2004} Expos\'e XIII, 2.10, Grothendieck showed that the specialization map is still valid when restricted on the tame fundamental groups: $$sp^t: \etpi(\widetilde{U}_K, \bar{x})\cong \pi_1^{\text{\'et}, t}(\widetilde{U}_K, \bar{x})\rightarrow \pi_1^{\text{\'et}, t}(U, \bar{y})$$ is surjective. 
        
       The $l$-adic local system of $U$ is still at most finite as before. Taking a finite \'etale cover if necessary so that the monodromy representation $\rho_l: \etpi(U, \bar{y})\rightarrow \GL_{2g}(\Z_l)$ factors through the tame fundamental group $\pi_1^{\text{\'et}, t}(U, \bar{y})$. By the same argument as in Theorem~\ref{cannot be lifted to char 0}, this gives a contradiction.

        More generally, let $Z$ be a smooth quasi-projective variety together with a nontrivial morphism $\varphi: Z\rightarrow I(x)$ to some isogeny leaf. And suppose that the $\varphi: Z\rightarrow \mA_{g, \bar{\F}_p}$ lifts to an abelian scheme  $\tilde{\varphi}: \widetilde{Z}\rightarrow \mA_{g, W(k)}$, with $\widetilde{Z}$ smooth and quasi-projective over $W(k)$. Bertini's theorem over DVR (see \cite{bertini_DVR} Theorem 0) implies that one can cut $\widetilde{Z}$ with hyperplanes to obtain a smooth subvariety of dimension one (need not to be proper) $\widetilde{C}\subseteq \widetilde{Z}$. Taking the reduction modulo $p$, this shows that $\widetilde{C}$ is the smooth lift of some  curve $C$ mapped to the isogeny leaf. And this gives a contradiction as we just showed.

For an arbitrary smooth variety, we do not know how to proceed: one expects the statement
to remain true, but another approach seems necessary due to the lack of a suitable specialization
map. This will not be pursued here, since our main objects are families of abelian varieties over
curves.

\end{remark}

We conclude the above discussion  together with Theorem~\ref{cannot be lifted to char 0} with the following result.

\begin{theorem}
      Let $Z$ be a smooth  quasi-projective or proper irreducible variety over $k$ and let $\varphi: Z\rightarrow I(x)$ be a nontrivial map, where $I(x)$ is some isogeny leaf. Then $\varphi: Z\rightarrow I(x)$ cannot be lifted to $W(k)$.
\end{theorem}



We also give the following conjecture for general smooth varieties.

\begin{conjecture}
        Let $Z$ be a smooth irreducible variety over $k$ and let $\varphi: Z\rightarrow I(x)$ be a nontrivial map, where $I(x)$ is some isogeny leaf. Then $\varphi: Z\rightarrow I(x)$ cannot be lifted to $W(k)$.
\end{conjecture}

\section{Positivity of Hodge bundle and Higgs semistability}

In this section, we review the first de Rham cohomology of an abelian scheme.

\subsection{Hodge bundle and Lie algebra}\label{section hodge bundle and lie algebra}

Let $S$ be a Noetherian scheme over a field $k$.
Let $f: \mX\rightarrow S$ be an abelian scheme of relative dimension $g$. Then the first de Rham cohomology $\mbH^1_{dR}(\mX/S)$ admits a Hodge filtration\[\Fil_h^\bullet: 0\subset \Fil^1\subset \Fil^0,\]with $\Fil^1\coloneq f_*\Omega_{\mX/S}^1$ and $\Fil^0=\mbH^1_{dR}(\mX/S)$. The Hodge bundle of $\mX/S$ is then defined to be the sheaf $f_*\Omega_{\mX/S}^1$, which is known to be locally free of rank $g$. One also has the graded piece $$\Gr^0\coloneq \Fil^0/\Fil^1\cong R^1f_*\mO_{\mX/S},$$ which is also locally free of rank $g$. 

On the other hand, one defines the Lie algebra $\Lie(\mX/S)$ of $\mX/S$ to be the dual of the Hodge bundle, i.e., $$\Lie(\mX/S)\coloneq (f_*\Omega_{\mX/S}^1)^\vee.$$ 
It is known that one has a canonical isomorphism of fppf-sheaves of $\mO_S$-modules
 \[R^1f_*\mO_{\mX/S}\overset{\sim}{\longrightarrow}\Lie(\mX^\vee/S).\]

 In particular, if $\mX/S$ is a principally polarized abelian scheme, one has a canonical isomorphism of locally free $\mO_S$-modules $$R^1f_*\mO_{\mX/S}\cong \Lie(\mX/S),$$ and hence a short exact sequence \[0\rightarrow f_*\Omega_{\mX/S}^1\rightarrow \mbH^1_{dR}(\mX/S)\rightarrow \Lie(\mX/S)\rightarrow 0.\]

Moreover, recall that the fine moduli space $\mA_{g,1, n}$ ($n\geq 3$) admits a minimal compactification $$\nu: \mA_{g,1,n}\rightarrow\mA_{g,1,n}^*$$ such that $\pi_*\omega_{\mA/\mA_{g,1,n}}$ extends to an ample line bundle $\mathcal{L}$ over $\mA_{g,1,n}^*$, i.e., $\nu^*\mathcal{L}=\pi_*\omega_{\mA/\mA_{g,1,n}}$. Here, $\pi: \mA\rightarrow\mA_{g,1,,n}$ is the universal abelian scheme and $\omega_{\mA/\mA_{g,1,n}}$ is the relative canonical line bundle. The compactification can be defined over $\Z[\xi_n, \frac{1}{n}]$ (see \cite{faltings_degeneration_1990} V. Theorem 2.5) and hence can be reduced modulo $p$ for any prime $p$ not dividing $n$. 

In particular, let $C$ be a smooth proper curve over $k$. And let $\varphi: C\rightarrow \mA_g$ be a nontrivial map, then the Seshadri’s criterion (see \cite{Positivity_in_Algebraic_Geometry}, Example 6.1.20) implies that $(\nu\circ\varphi)^*\mathcal{L}=\varphi^*\pi_*\omega_{\mA/{\mA_{g,1,n}}}$ has positive degree. As $\pi_*\omega_{\mA/{\mA_{g,1,n}}}$ is locally free, one has $$f_*\omega_{\mX/C}\cong \varphi^*\pi_*\omega_{\mA/{\mA_{g,1,n}}},$$ where $f: \mX\rightarrow C$ is the induced universal family. This shows that $\deg(f_*\Omega_{\mX/C}^1)>0$  if and only if $\varphi: C\rightarrow\mA_g$ is nontrivial if and only if $\mX/C$ is a non-isotrivial family of abelian varieties.

\subsection{Higgs semistability}

The de Rham cohomology $\mbH^1_{dR}(\mX/S)$ carries an integral connection over $S$, called the Gauss--Manin connection and denoted by $\nabla_{GM}$. The connection $\nabla_{GM}$ satisfies the Griffiths transversality, i.e., $$\nabla_{GM}(\Fil^1)\subset \Fil^0\otimes \Omega^1_S.$$ 

The induced map on graded pieces$$\theta: \Gr^1\rightarrow \Gr^0\otimes \Omega^1_S$$  is an $\mO_S$-linear morphism and is called the Kodaira--Spencer map. This terminology is justified since locally $\theta$
can be identified with cup product with the Kodaira–-Spencer class (see \cite{katz_nilpotent_1970} Theorem 3.5).

We  denote $\mE$ for the Hodge bundle $f_*\Omega^1_{\mX/S}$ and by the discussion of \ref{section hodge bundle and lie algebra}, one sees that $\mE^\vee$ is identified with the Lie algebra. The flatness of $\nabla_{GM}$ implies that $\theta$ is a Higgs field, and hence $(\mE\oplus \mE^\vee, \theta)$ gives a Higgs bundle over $S$. The graded module  of  $(\mbH^1_{dR}(\mX/S), \nabla_{GM})$ is then the Higgs bundle $(\mE\oplus \mE^\vee, \theta)$ with $$\theta: \mE\rightarrow\mE^\vee\otimes \Omega_{C/k},$$ and $$\theta(\mE^\vee)=0.$$

From now on, assume $C$ is a smooth proper curve over a field $k$ and $\mX\rightarrow C$ is a principally polarized abelian scheme over $C$. We refer to Appendix~\ref{appendix} for the necessary background for Higgs semistability (polystability).

We recall the following non-abelian Hodge theory result of Simpson. 

\begin{theorem}[Simpson's correspondence, \cite{Simpson1992HiggsBA} Corollary 1.3]\label{simpson}
  Let $C$ be a complex smooth projective curve. Then there exists an equivalence of categories between the category of semisimple flat bundles over $C$ and the category of polystable Higgs bundles of slope zero.
\end{theorem}

In particular, if $f: \mX\rightarrow C$ is an abelian scheme over the complex curve $C$. The monodromy representation associated to $R^1f^{\text{an}}_*\underline{\C}$ is always semisimple, and hence the Higgs bundle $(\Gr^\bullet\mbH^1_{dR}(\mX/C), \theta)$ is polystable of slope zero.

However, the analogue in positive characteristic case does not hold in general, as one will see later that the Moret-Bailly is a counter-example. In order to have the (semi)stability result, one needs to assume a $W_2(k)$-lift for the abelian scheme.

\begin{theorem}[\cite{ogus_nonabelian_2007}, Proposition 4.19]\label{ogus Vologodsky}
Let $C$ be a smooth projective curve of genus $g(C)$ over a field $k$ of characteristic $p>0$. Suppose the base change of Frobenius  $C^{(p)}$ over $k $ has a lift $\widetilde{C^{(p)}}/W_2(k)$. Let $(M, \nabla, \Fil^\bullet, \Phi)$ be a Fontaine module of weight belonging to $[0, n]$. Assume that \[n(\rank(M)-1)\max\{2g(C)-2, 1\})<p-1.\]Then $(\text{Gr}_{\Fil^\bullet}^\bullet M, \theta)$ is a semistable Higgs bundle over $C$. 
\end{theorem}

In particular, let $f: X\rightarrow C$ be a  smooth proper morphism and let $$\widetilde{f}: \widetilde{ X^{(p)}}\rightarrow \widetilde{C^{(p)}}$$ be a lift to $W_2(k)$. If $M=\mathbb{H}^n_{dR}(X/C)$
is a Fontaine module together with the Gauss--Manin connection $\nabla_{GM}$, Hodge filtration $\Fil_{h}^\bullet$ and a relative Frobenius $\Phi$. Then the associated Higgs bundle $(\text{Gr}_{\Fil_{h}^\bullet}^\bullet M, \theta)$ is semistable if $$n(\rank(M)-1)\max\{2g(C)-2, 1\})<p-1.$$

\begin{remark}\label{lifting remark}
Let $f: X\rightarrow C$ is a smooth proper morphism of relative dimension $g$. If $f: X\rightarrow C$ lifts to $W_2(k)$, say $\widetilde{f}: \widetilde{X}\rightarrow\widetilde{C}$, then the base change of $\widetilde{f}$ along the Frobenius of 2-truncated Witt vectors $F_{W_2(k)}: W_2(k)\rightarrow W_2(k)$ gives a lift of $X^{(p)}\rightarrow C^{(p)}$ to $W_2(k)$. 
\end{remark}

In \cite{lan_semistable_2017}, Lan--Sheng--Zuo developed the theory of Higgs--de Rham flow. In particular, they generalized Ogus--Vologodsky's result Theorem~\ref{ogus Vologodsky} to higher dimensional varieties and removed the restrictions on $p$.

\begin{theorem}[\cite{lan_semistable_2017}, Proposition 6.3]\label{lan sheng zuo}
  Let $X/k$ be a smooth projective variety and let $(M, \nabla, \Fil^\bullet, \Phi)$ be a strict $p$-torsion Fontaine module (with respect to some $W_2(k)$-lifting of $X^{(p)}$). Then the graded Higgs bundle $(\text{Gr}_{\Fil^\bullet}^\bullet M, \theta)$ is Higgs semistable of slope zero.

\end{theorem}

Returning to the abelian scheme case $f:\mX\rightarrow C$. If its Frobenius base change $X^{(p)}\rightarrow C^{(p)}$
admits a $W_2(k)$-lifting to $\widetilde{X^{(p)}}\rightarrow \widetilde{C^{(p)}}$. Then $(\mbH^1_{dR}(\mX/C), \nabla_{GM}, \Fil^\bullet_h, \Phi)$ is a strict $p$-torsion Fontaine module.
We summarize the above discussion with the following result.

\begin{theorem}[Ogus--Vologodsky \cite{ogus_nonabelian_2007}, Lan--Sheng--Zuo \cite{lan_semistable_2017}]\label{higgs semistable}
  Let $C$ be a smooth proper curve over field $k$ of characteristic $p>0$ and let $f: \mX\rightarrow C$ be a family of principally polarized  abelian varieties of dimension $g$ over $C$. If $f: \mX\rightarrow C$ can be lifted to $W_2(k)$, then the Higgs bundle $(\mE\oplus\mE^\vee, \theta)$ is Higgs semistable of slope zero.
\end{theorem}

\subsection{Positivity results and an Arakelov inequality}

Let $\mE$ be a vector bundle over a smooth proper curve $C$ over a field $k$. $\mE$ admits a (unique) Harder-Narasimhan filtration\[0=\mE_0\subset\mE_1\subset\cdots\mE_n=\mE\]such that the successive quotients $\mE_i/\mE_{i-1}$ are semistable vector bundles such that $$\mu(\mE_1)>\mu(\mE_2/\mE_1)>\cdots>\mu(\mE_n/\mE_{n-1}).$$In particular, $\mE_1$ is semistable. We write $\mu_{\max}(\mE)\coloneq \mu(\mE_1)$ and $\mu_{\min}(\mE)\coloneq \mu(\mE_n/\mE_{n-1})$. 

We call the vector bundle $\mE$ to be ample (resp. nef) if the tautological bundle $\mO_{\mathbb{P}(\mE)}{(1)}$ is ample (resp. nef) over $\mathbb{P}(\mE)$. It is known that if $\text{char}(k)>0$, $\mE$ is ample (resp. nef) if and only if $\mu_{\min}(\mE)>0$ (resp. $\mu_{\min}(\mE)\geq 0$).

If $\text{char}(k)=p>0$, such a criterion might fail. In order to fix this, one needs to define \[\bar{\mu}_{\min}(\mE)=\lim_{n\rightarrow \infty}\frac{\mu_{\min}((F_C^n)^*\mE)}{p^n},\]\[\bar{\mu}_{\max}(\mE)=\lim_{n\rightarrow \infty}\frac{\mu_{\max}((F_C^n)^*\mE)}{p^n},\] where $F_C$ is the absolute Frobenius of $C$. These two limits turn out to be well-defined rational numbers by \cite{semistable_sheaves_langer} Theorem 2.7. And it was shown by Barton (see \cite{Barton_tensor_ample}) that $\mE$ is ample (resp. nef) if and only if $\bar{\mu}_{\min}(\mE)>0$ (resp. $\bar{\mu}_{\min}(\mE)\geq 0$). 

Moreover, they always satisfy the relations \[\mu_{\max}(\mE)\leq \bar{\mu}_{\max}(\mE),\]\[ \mu_{\min}(\mE)\geq \bar{\mu}_{\min}(\mE).\]

\begin{proposition}\label{3.7}
Let $\mE$ be a vector bundle of rank $g$ over a curve $C$ of genus $g(C)$. Then if $\mu_{\min}(\mE)>0$, one has $\bar{\mu}_{\min}>0$ for $p\gg 0$. Similarly, if $\mu_{\max}(\mE)<0$, then $\bar{\mu}_{\max}(\mE)<0$ for $p\gg 0$.
\end{proposition}

\begin{proof}
By \cite{semistable_sheaves_langer} Corollary 6.2, \[\max\{\bar{\mu}_{\max}(\mE)-\mu_{\max}(\mE), \mu_{\min}(\mE)-\bar{\mu}_{\min}(\mE)\}\leq \frac{g-1}{p}\max\{0, 2g(C)-2\},\]

If $g(C)=0$ or 1, the proposition follows easily for any $p$. So, we may suppose  $g(C)\geq 2$. So, \[
\bar{\mu}_{\min}(\mE)\geq \mu_{\min}(\mE)-\frac{g-1}{p}(2g(C)-2).\tag{$*$}
\] On the other hand, since $\mu_{\min}(\mE)>0$, one has $\mu_{\min}(\mE)\geq \frac{1}{g}$. So, for any $$p\geq g(g-1)(2g(C)-2)+1,$$ $(*)$ shows that $\bar{\mu}_{\min}(\mE)>0$.

The other case is similar.

\end{proof}

The stability-type results of the Higgs bundle $(\mE\oplus \mE^\vee, \theta)$ impose strong restrictions on the positivity of the Hodge bundle $\mE$. 
If $k=\C,$ Griffiths \cite{Griffiths1970} proved that $\mE$ is nef by using analytic tools. Since we are interested in the Higgs bundle, we note that this is also a consequence of Simpson's correspondence Theorem~\ref{simpson}. Indeed, if $\mu_{\min}(\mE)<0$, then $\mE/\mE_{n-1}$ is a semistable bundle of slope less than zero. Taking the dual, one has a short exact sequence \[0\rightarrow (\mE/\mE_{n-1})^\vee\rightarrow \mE^\vee\rightarrow \mE_{n-1}^\vee\rightarrow 0.\]Since $\theta(\mE^\vee)=0$, $((\mE/\mE_{n-1})^\vee, 0)\subset (\mE\oplus\mE^\vee, \theta)$ is a Higgs subsheaf. But this contradicts the polystability of $(\mE\oplus\mE^\vee, \theta)$. Generally, let $k$ is a field of characteristic zero. Then the nefness of $\mE$ follows from the fact that semistability is preserved by field extension (see \cite{Huybrechts_Lehn_2010} Corollary 1.3.8).

If $\text{char}(k)>0$, the nefness of the Hodge bundle might fail. We will show in the next section that any curve inside the isogeny leaf gives a counter-example for the nefness of the Hodge bundle. However, if the abelian scheme $\mX/C$ has a $W_2(k)$-lift, $(\mE\oplus \mE^\vee, \theta)$ is Higgs semistable of slope zero by Theorem~\ref{higgs semistable}. In particular, one has $\mu_{\max}(\mE^\vee)\leq0$ by the same argument as in the complex case. This in turn gives $\mu_{\min}(\mE)=-\mu_{\max}(\mE^\vee)\geq0$.

Assuming the $W_2(k)$-lifting, one can also get the following Arakelov inequality for the Hodge bundle, which is well-known in the complex case.

\begin{theorem}[Arakelov inequality]\label{arakelov}
  Let $f: \mX\rightarrow C$ be a family of principally polarized  abelian varieties of relative dimension $g$ over a proper smooth curve. Assume that $f$ can be lifted to $W_2(k)$, then \[\deg f_*\Omega_{\mX/C}^1\leq \frac{g}{2}(2g(C)-2).\]
\end{theorem}

\begin{proof}
  Given the family $f: \mX\rightarrow C$ which can be lifted to $W_2(k)$, let us  consider the Higgs bundle $(\mE\oplus\mE^\vee, \theta)$ which is Higgs semistable by Theorem~\ref{higgs semistable}, where $\mE\cong f_*\Omega_{\mX/C}^1$.

  Let $K\coloneq \Ker(\theta: \mE\rightarrow \mE^\vee\otimes\Omega_{C/k})$. 
Then we have exact sequences\[0\rightarrow K\rightarrow \mE\rightarrow F\rightarrow 0. \]

Let $F'$ be the saturation of $F$ in $\mE^\vee\otimes\Omega_{C/k}$, i.e., it is a subbundle of $\mE^\vee\otimes\Omega_{C/k}$ with $\rank(F)=\rank(F'), \deg(F)\leq \deg(F')$. We then have an exact sequence of vector bundles over $C$:

\[0\rightarrow F'\rightarrow \mE^\vee\otimes\Omega_{C/k}\rightarrow G\rightarrow 0. \tag{$*$}\]

We then have\[\deg(\mE)=\deg(K)+\deg(F),\]
and $$\deg(\mE^\vee\otimes \Omega_{C/k})=g(2g(C)-2)-\deg(\mE)=\deg(F')+\deg(G)\leq \deg(F)+\deg(G).$$

We have\[2\deg(\mE)\leq g(2g(C)-2)+\deg(K)-\deg(G). \tag{$**$}\]

$(K, 0)$ gives a Higgs subbundle of $(\mE\oplus\mE^\vee, \theta)$ obviously. Since $(\mE\oplus\mE^\vee, \theta)$ is Higgs semistable, we know $\deg(K)\leq 0$.

On the other hand, one can take the dual and then tensor with $\Omega_{C/k}$ to the exact sequence $(*)$. So, we get an exact sequence of vector bundles\[0\rightarrow G^\vee\otimes \Omega_{C/k}\rightarrow \mE\rightarrow F'^\vee\otimes \Omega_{C/k}\rightarrow 0.\]

Notice that the map $$\mE\rightarrow F'^\vee\otimes \Omega_{C/k}\cong \Hom(F', \Omega_{C/k})$$ is locally given by \[x\mapsto (\sum_i\phi_i\otimes \omega_i\mapsto \sum_i\phi_i(x)\cdot \omega_i).\]

So, $G^\vee\otimes \Omega_{C/k}$ is locally given by the sections $x\in \mE$ such that\[\forall y\in \mE, \text{KS}(x, y)=0,\]
where $\text{KS}: \mE\otimes \mE\rightarrow \Omega_{C/k}$ is the Kodaira--Spencer map induced by $\theta$. $\text{KS}$ is symmetric after identification with the principal polarization (for example, see \cite{KaiwenLan} Theorem 2.3.4.2). So, $G^\vee\otimes \Omega_{C/k}$ can also be identified with $\Ker(\theta)$ and  $(G^\vee\otimes \Omega_{C/k}, 0)$ then gives a Higgs subbundle of $(\mE\oplus\mE^\vee, \theta)$.

Since $(\mE\oplus\mE^\vee, \theta)$ is Higgs semistable, we know $$\deg(G^\vee\otimes\Omega_{C/k})=\rank(G)(2g(C)-2)-\deg(G)\leq 0.$$



Take these inequalities back to $(**)$, we see that\[\deg(\mE)\leq \frac{g-\rank(G)}{2}(2g(C)-2)\leq g(g(C)-1).\]

\end{proof}

\begin{remark}\label{a nonlift remark}
Let $C$ be a curve of genus $g(C)=0$ or 1 and let $\phi: C\rightarrow \mA_g$ be a nontrivial map. Theorem~\ref{oort} and Theorem~\ref{trivial l-adic monodromy} show that such curve has to lie in some isogeny leaf. And Theorem~\ref{arakelov} shows that such family cannot be lifted to $W_2(k)$, as one should have $\deg f_*\Omega_{\mX/C}^1>0$. This gives the first evidence for the non-liftability of subvarieties in isogeny leaves.

\end{remark}







\subsection{Moret--Bailly family}In this section, we give an introduction to the Moret--Bailly family, which motivates this paper and exhibits the existence of smooth proper curves inside isogeny leaves.

In \cite{AST_1981__86__R1_0}, Moret--Bailly constructed a family of supersingular principally polarized abelian surfaces over $\mathbb{P}^1$ in characteristic $p>0$, which we  call the Moret--Bailly family. Briefly, the family was constructed as follows: let $A_0\cong E_1\times E_2$ be a supersingular abelian surface over $k$, i.e., $E_1, E_2$ are supersingular elliptic curves over $k$. Let $O_i\in E_i ,i=1, 2$ be the zero points. Consider the line bundle $\mathcal{L}_0\coloneq \mO_{A_0}(E_1\times \{O_2\}+\{O_1\}\times E_2)^{\otimes p}$, which can be shown to be ample and thus gives a polarization on $A_0$. One can calculate that $K(\mathcal{L}_0)$
contains a subgroup scheme $H_0\cong \alpha_p$. In \cite{AST_1981__86__R1_0}, Moret--Bailly showed that $A\cong A_0/H_0$ is still superspecial and the line bundle $\mathcal{L}_0$ descends to an ample line bundle $\mathcal{L}$ over $A$ with $K(\mathcal{L})\cong \alpha_p\times \alpha_p$. 

We write $K\coloneq K(\mathcal{L})$. Then $$K\cong \alpha_p\times \alpha_p\cong \spec(k[a]/a^p)\times \spec(k[b]/b^p).$$
Consider the projective line $\mathbb{P}_k^1$ with homogeneous coordinates $U, V$. Define $H$ to be the subgroup scheme contained in $K\times \mathbb{P}^1_k\cong \underline{\spec}_{\mO_{\mathbb{P}^1_k}}(\mO_{\mathbb{P}^1_k}[\alpha, \beta]/(\alpha^p, \beta^p))$ that is defined by $U\alpha-V\beta$. $H$ is then a finite flat group scheme over $\mathbb{P}^1_k$ which is locally isomorphic to $\alpha_p$.

Finally, one considers the principally polarized abelian scheme $\mX\cong (A\times \mathbb{P}^1_k)/H$ over $\mathbb{P}^1_k$. This is called the Moret--Bailly family.
 Moreover, Moret--Bailly also gives a characterization for its Lie algebra:

\begin{theorem}[Moret--Bailly, \cite{AST_1981__86__R1_0}]\label{lie algebra for moret bailly}
  $\Lie(\mX/\mathbb{P}^1_k)\cong \mO(-p)\oplus \mO(1)$.
\end{theorem}

Remark~\ref{a nonlift remark} then implies that the Moret--Bailly family does not admit a lift to $W_2(k)$. 

More generally, let $\mS_2$ be the induced reduced subscheme of the supersingular locus of $\mA_2$.
In \cite{Oortsubvar} Corollary 4.7, Oort showed that $\mS_2$ contains a proper rational curve. In \cite{Families_of_supersingular_abelian_surfaces}, Katsura and Oort showed that every irreducible component of the supersingular locus $\mS_2$ is the image of a Moret--Bailly family. This furthermore implies that every isogeny leaf in $\mS_2$ is the image of a Moret--Bailly family by Theorem~\ref{almostproduct}. So generally, any isogeny leaf in $\mS_2$ does not admit a lift to $W_2(k)$.

\begin{remark}
  Let $f: \mathbb{P}^1\rightarrow \mA_2$ be the morphism giving rise to the Moret--Bailly family. As the Moret--Bailly family $\mX/\mathbb{P}^1$ has Hodge bundle $\mE\cong \mO(p)\oplus \mO(-1)$, it turns out $$H^0(C, f^*T_{\mA_2})=\Hom_{\mO_C}(\text{Sym}^2\mE, \mO_C)=k$$ is nontrivial. This implies the Moret--Bailly family is non-rigid. More generally, this shows that $\mS_2$ is non-rigid inside of $\mA_2$. But this should not be surprising, as one take $\mS_2$ to have the  induced reduced scheme structure. In particular, isogeny leaves are not well suited for studying
(geometric) infinitesimal rigidity of subvarieties.

\end{remark}

\section{Non-liftability of subvarieties of isogeny leaves}

In this section, we prove the main theorem.

\begin{theorem}\label{maintheorem}
   Let $x=[(A, \lambda)]\in \mA_{g}(k)$ be a non-ordinary principally polarized abelian variety over $k$ and let $I(x) $ denote the isogeny leaf through $x$. Let $C$ be a proper and smooth curve over $k$. If $\varphi: C\rightarrow I(x)$ is a nontrivial map, then $\varphi$ cannot be lifted to $W_2(k)$. 
\end{theorem}

Our strategy is to show the following result about the positivity of the Hodge bundle, which  implies the non-liftability of the family by Theorem~\ref{higgs semistable}.

\begin{theorem}\label{maintheorem 2}
   Let $x=[(A, \lambda)]\in \mA_{g}(k)$ be a non-ordinary principally polarized abelian variety over $k$ and let $I(x) $ denote the isogeny leaf through $x$. Let $C$ be a proper and smooth curve over $k$. If $\varphi: C\rightarrow I(x)$ is a nontrivial map, then $\mu_{\max}(\Lie(\mX/C))>0$. 
\end{theorem}

\begin{proof}[Proof of \ref{maintheorem}]
  If $\varphi$ can be lifted to $W_2(k)$, Theorem~\ref{higgs semistable} implies that $(\mE\oplus\mE^\vee, \theta)$ is Higgs semistable of slope zero, where $\mE=f_*\Omega_{\mX/C}^1$. In particular, we have $\mu_{\max}(\Lie(\mX/C))\leq 0$, which contradicts Theorem~\ref{maintheorem 2}.
\end{proof}

It remains to prove \ref{maintheorem 2}.
Let $\varphi: C\rightarrow I(x)$ be a nontrivial map from a smooth proper irreducible curve $C$ to some isogeny leaf $I(x)$. By the discussion in Remark~\ref{remark 1}, there exists an \'etale cover $\pi: D\rightarrow C$ and a polarized abelian variety $(B_0, \mu)$ over $k$, such that there exists an isogeny of abelian $D$-schemes  $\phi: (B_0, \mu)\times_{\spec(k)} D\rightarrow \mX\times_CD$  of local-local type, where $f:\mX\rightarrow C$ is the induced abelian scheme over $C$.

Since $\pi$ is \'etale, one has a natural isomorphism as locally free $\mO_D$-modules\[\Lie(\mX\times_CD/D)\overset{\sim}{\longrightarrow}\pi^*\Lie(\mX/C).\]
And the Harder–Narasimhan filtration of $\Lie(\mX/C)$\[0=\mF_0\subset \mF_1\subset\cdots\subset\mF_n=\Lie(\mX/C)\]
gives rise to the Harder–Narasimhan filtration of $\Lie(\mX\times_CD/D)$ \[0=\pi^*\mF_0\subset \pi^*\mF_1\subset\cdots\subset \pi^*\mF_n=\Lie(\mX\times_CD/D).\]
One can find the details in Appendix~\ref{appendix}. In particular, $\mu_{\max}(\Lie(\mX/C))>0$ if and only if $\mu_{\max}(\Lie(\mX\times_CD/D))>0$. So, without loss of generality, we can replace $C$ by $D$ and assume that $\mX/C$ admits an isogeny from a constant abelian scheme over $C$.

The isogeny
$\phi$ then induces a map between Lie algebras\[\Lie(\phi): \Lie(B_0\times C/C)\rightarrow \Lie(\mX/C).\]
Notice that $\Lie(B_0\times C/C)\cong \Lie(B_0/k)\otimes_k\mO_C\cong \mO_C^{\oplus g}$ is a trivial vector bundle, hence semistable of slope 0.

\subsection{Case I: $\Lie(\phi) = 0$ }\label{case 1}

\begin{definition}[Restricted Lie algebras]
  Let $R$ be a commutative ring of characteristic $p>0$. A restricted Lie algebra over $R$ is a Lie algebra $(L, [ -, - ])$ consisting a finite locally free $R$-module $L$ and a Lie bracket $[- , -]: L\times L\rightarrow L$ together with a $p$-mapping $(-)^{[p]}: L\rightarrow L, x\mapsto x^{[p]}$ such that:

  1. $\ad(x^{[p]})=(\ad x)^p$, where $\ad$ is the adjoint representation of $L$ on itself defined by $\ad(x)(y)=[x, y]$;

  2. $(tx)^{[p]}=t^p x^{[p]}$ for any $t\in R$ and $x\in L$;

  3. $(x+y)^{[p]}=x^{[p]}+y^{[p]}+\sum_{i=1}^{p-1}s_i(x, y)$ for any $x, y\in L$, where $s_i(x, y)$ is $\frac{1}{i}$ times the coefficient of $t^{i-1}$ in the formal expression $(\ad(tx+y))^{p-1}(x)$.

\end{definition}

Generally, let $S$ be a scheme of characteristic $p>0$. A sheaf of restricted Lie algebras over $S$ is a sheaf of $\mO_S$-modules $\mathcal{L}$ together with a Lie bracket $[-, -]: \mathcal{L}\otimes_{\mO_S}\mathcal{L}\rightarrow \mathcal{L}$ and a $p$-th power map $(-)^{[p]}: \mathcal{L}\rightarrow \mathcal{L}$ such that for some affine covering $S=\cup_{i\in I}\spec(R_i)$, $(\mathcal{L}(\spec(R_i)), [- , -], (-)^{[p]})$ is a restricted Lie algebra over $R_i$.

Let $G/S$ be a commutative group scheme over a base $S$ of characteristic $p>0$. $\Lie(G/S)$ then gives a sheaf of restricted Lie algebras over $S$ (see \cite{SGAIII_VII_A}, Expos\'e VIIA, Chapter 6). 
And since $G$ is commutative,  the Lie bracket is trivial.

A $p$-mapping on $\Lie(G/S)$ is equivalent to an $\mO_S$-linear morphism $$\phi: F_S^*\Lie(G/S)\rightarrow \Lie(G/S),$$ defined by $\phi(F_S^*x)=x^{[p]}$. This is well-defined, as $F_S^*:\Lie(G/S)\rightarrow F_S^*\Lie(G/S)$ is $p$-linear.
As a result, the category of sheaves of commutative restricted Lie algebras over $S$ is equivalent to the category of pairs $(\mF, \phi)$ consisting of a locally free $\mO_S$-modules and an $\mO_S$-linear morphism $\phi: F_S^*\mF\rightarrow\mF$.

On the other hand, let $F=F_{G/S}: G\rightarrow G^{(p)}$ be the relative Frobenius morphism. Then the classical infinitesimal study on  group schemes shows that $\Lie(G[F]/S)\cong \Lie(G/S)$, where $G[F]=\Ker(F)$. This motivates the following equivalence:

\begin{theorem}[\cite{SGAIII_VII_A}, Expos\'e VIIA, Chapter 7]\label{restricted lie algebra equivalence}
  Let $S$ be a scheme of characteristic $p>0$, then there is an equivalence between the category of finite flat group schemes of height at most one over $S$ (i.e., $F_{G/S}: G\rightarrow G^{(p)}$ is trivial) and the category of sheaves of commutative restricted Lie algebras over $S$ given by $G\rightarrow \Lie(G/S)$. 
\end{theorem}


\begin{lemma}\label{main lemma 1}
  Let $f:X\rightarrow Y$ be an isogeny of abelian schemes over a scheme $S$ with characteristic $p>0$. Then $\Lie(f): \Lie(X/S)\rightarrow \Lie(Y/S)$ is trivial if and only if $f=g\circ F_{X/S}$ with $F_{X/S}$ the relative Frobenius morphism.
\end{lemma}
\begin{proof}
  One direction is clear. Conversely, it suffices to show $\Ker(F)\subseteq \Ker(f)$.



The isogeny $f: X\rightarrow Y$ restricts to a morphism of finite flat group schemes over $S$ of height at most one $f: X[F]\rightarrow Y[F]$. As$$\Hom_S(X[F], Y[F])\rightarrow \Hom_{\mO_S}(\Lie(X[F]/S), \Lie(Y[F]/S))\cong \Hom_{\mO_S}(\Lie(X/S), \Lie(Y/S)),$$given by $$f\mapsto \Lie(f)$$ is an isomorphism by Theorem~\ref{restricted lie algebra equivalence}.
We must have $f:X[F]\rightarrow Y[F]$ is trivial by the assumption that $\Lie(f)=0$. As a result, we have $\Ker(F)\subseteq \Ker(f)$.
\end{proof}

By Lemma~\ref{main lemma 1} and the finiteness of $\phi$, one must have $\phi=\phi'\circ F^m_{B_0\times C/C}$ for some $m\geq 0$ with $$d\phi': \Lie((B_0\times_k C)^{(p^m)}/C)\rightarrow\Lie(\mX/C)$$  nontrivial. Moreover, as the  Frobenius is stable under base change, we have $$(B_0\times_{\spec(k)}C)^{(p^m/C)}\cong B_0^{(p^m/k)}\times_{\spec(k)}C.$$ So, we still have an isogeny  $\phi'$ with source a constant group scheme.

\begin{remark}
  Let us consider the Lie algebra of Moret--Bailly family $f: \mX\rightarrow \mathbb{P}^1$ and give a coarse estimation. By the classification of vector bundles over $\mathbb{P}^1$, $\Lie(\mX/\mathbb{P}^1)\cong \mO(n)\oplus \mO(m)$ for some $n, m\in \Z$. As $\Lie(\mX/\mathbb{P}^1)$ has slope strictly less than zero, we can assume $n<0$ and $n+m<0$. We would like to show $m\geq 1$.
  
  Recall that $\mX$ is constructed to be the quotient of $E_0\times E_0\times \mathbb{P}^1$ by the group scheme $$H\subset \alpha_p\times \alpha_p\times \mathbb{P}^1\cong \underline{\spec}_{\mO_{\mathbb{P}^1}}(\mO_{\mathbb{P}^1}[\alpha, \beta]/(\alpha^p, \beta^p))$$ given by the closed locus $U\alpha-V\beta$, with $U, V$ are the homogeneous coordinates of $\mO_{\mathbb{P}^1}$. 

As $H$ has Lie algebra rank 1,  the isogeny $\phi: E_0\times E_0\times \mathbb{P}^1\rightarrow \mX$ induces a nontrivial Lie algebra morphism and hence $m\geq 0$. Moreover, as $H$ is non-constant, each factor of $$\Lie(\phi): \Lie(E_0)\otimes \mO\rightarrow \Lie(\mX/C)$$ must be nonzero. If $m=0$, we then have a surjection \[\Lie(\phi): \mO\oplus \mO\twoheadrightarrow \mO,\] with isomorphism $\gamma_i$ on each factor $\mO\overset{\sim}{\longrightarrow}\mO$ and hence $\gamma_i\in \Gamma(\mathbb{P}^1, \mO)^*=k^*, i=1, 2$. In particular, the kernel of $\Lie(\phi)$ is given by a $k$-line $(u, \gamma u)\in \mO\oplus\mO$, with $\gamma=\gamma_2^{-1}\cdot \gamma_1$. So, the group scheme corresponds to $\Ker(\Lie(\phi))$, which is  $H$, should be given by a ``$k$-line'' $V(\alpha-\gamma\beta)\subset\alpha_p\times\alpha_p\times \mathbb{P}^1$. This then gives a contradiction with the non-constancy of $H$.

This computation provides intuition for the reduction step below.
\end{remark}

\subsection{Case II: $\Lie(\phi)\neq 0$ }\label{case 2}

In this section, we   assume $\Lie(\phi)\neq 0$ and $$\mu_{\max}(\Lie(\mX/C))\leq 0.$$In other words, the Harder-Narasimhan filtration of $\Lie(\mX/C)$ is of the form $$0=\mF_0\subset \mF_1\subset \cdots\subset \mF_n=\Lie(\mX/C),$$
with $\mF_1$ is semistable of slope $\leq 0$. We  write $\mbL\coloneq  \mF_1$, and hence $\Lie(\mX/C)/\mbL$ is a vector bundle that is strictly negative, i.e., $\mu_{\max}(\mE/\mbL)<0$.

We note that $\Lie(\phi)\neq 0$ would imply $\mbL$ is semistable of slope 0. Indeed, the composition of maps\[\mO_C^{\oplus g}\overset{\Lie(\phi)}{\longrightarrow}\Lie(\mX/C)\twoheadrightarrow \Lie(\mX/C)/\mbL\] has to be  trivial by the negativity of $\Lie(\mX/C)/\mbL$. 

In particular,  $\Lie(\phi)$ can  be factored as\[\mO_C^{\oplus g}\overset{\Lie(\phi)}{\longrightarrow}{\mbL}\hookrightarrow \Lie(\mX/C),\]
where the first arrow is a nontrivial map between semistable vector bundles of slope 0.

By Theorem~\ref{higgs semistable abelian category}, $\Ker(\mO_C^{\oplus g}\rightarrow \Lie(\mX/C))=\Ker(\mO_C^{\oplus g}\rightarrow \mbL)$ is a subbundle of slope 0 of $\mO_C^{\oplus g}$.

\begin{lemma}\label{lemma 2}
Let $\mE\subseteq \mO_C^{\oplus n}$ be a subbundle of slope 0, then $\mE$ is also free, of the form $\mE\cong \mO_C^{\oplus r}$ for some $r$. In particular, any short exact sequence of degree 0 vector bundles\[0\rightarrow\mE\rightarrow \mO_C^{\oplus  n}\rightarrow \mF\rightarrow 0\]splits.
\end{lemma}





\begin{proof}
  By Theorem~\ref{higgs semistable abelian category},  the category of semistable vector bundles of fixed slope $\mu$ is an abelian category. Consider the abelian category of semistable vector bundles of slope 0, it has a simple object $\mO_C$. And the lemma follows from the fact in category theory that if $\mE$ is a finite direct sum of simple objects $\mF_i$ in some abelian category, any subobject of $\mE$ is a direct summand and isomorphic to a direct sum of some of the simple objects $\mF_i$.
\end{proof}

\begin{lemma}\label{lemma 3}
  Let $f: V\rightarrow W$ be a map between commutative  restricted Lie algebras over $S$, then the kernel (resp. image) $\Ker(f)$ (resp. $\Image(f)$) has a unique structure of commutative restricted Lie algebras over $S$. 
\end{lemma}

\begin{proof}

To give a morphism of commutative restricted Lie algebras $f: V\rightarrow W$ is the same to give a commutative diagram

\[\begin{tikzcd}
	{F_C^*V} & {F_C^*W} \\
	V & W
	\arrow["{F_C^*f}", from=1-1, to=1-2]
	\arrow[from=1-1, to=2-1]
	\arrow[from=1-2, to=2-2]
	\arrow["f", from=2-1, to=2-2]
\end{tikzcd}.\]
Since $C$ is smooth over $k$, the absolute Frobenius $F_C$
is flat, and the commutative diagram naturally extends to a commutative diagram

\[\begin{tikzcd}
	{F_C^*\Ker(f)} & {F_C^*V} & {F_C^*W} \\
	{\Ker(f)} & V & W
	\arrow[hook, from=1-1, to=1-2]
	\arrow[from=1-1, to=2-1]
	\arrow["{F_C^*f}", from=1-2, to=1-3]
	\arrow[from=1-2, to=2-2]
	\arrow[from=1-3, to=2-3]
	\arrow[hook, from=2-1, to=2-2]
	\arrow["f", from=2-2, to=2-3]
\end{tikzcd}.\]
And this shows that $\Ker(f)$ has a natural structure of commutative restricted Lie algebra. Similar for the image.

\end{proof}

We remark that though the category of restricted Lie algebras has images, it might not corresponds to a group subscheme. The reason is simple: closed group subscheme $H\subseteq G$ always imply a cokernel $G/H$ represented by a finite flat group scheme, but the cokernel of a subsheaf $\Image(f)\subseteq W$ is not always 
torsion-free. 

However, Lemma~\ref{lemma 2} implies that the kernel of $\Lie(\phi)$ is then a sub-bundle of $\mO_C^{\oplus g}$, i.e., it has cokernel torsion-free. It also has a natural structure of restricted Lie algebras by Lemma~\ref{lemma 3}. Similarly, as $\Lie(\phi): \mO_{C}^{\oplus g}\rightarrow \Lie(\mX/C)$ eventually has torsion-free cokernel, the image of $\Lie(\phi)$ would correspond to group subscheme of $\mX[F]$.
We are now going to show such Lie algebras are actually constant.

\begin{lemma}\label{lemma 4}
  Let $\mV$ be a trivial restricted Lie algebra, i.e., $\mV\cong \mO_C^{\oplus n}$.
  Then, it corresponds to a constant finite flat group scheme  $\mG$ of height one  over $C$, i.e., $\mG\cong G\times_kC$ for some finite flat group scheme $G$ over $k$; 

\end{lemma}

\begin{proof}
It suffices to show $\V\coloneq \Gamma(C, \mV)$ is a restricted Lie algebra over $k$ with identification $\mV\otimes \mO_C\cong \V$ as restricted Lie algebras. Indeed, let $G_0$ to be the correspondent group scheme of $\V$. Then the conclusion follows by the identification $\Lie(G_0\times C/C)\cong \Lie(G_0/k)\otimes_k \mO_C\cong \V\otimes \mO_C\cong\mV$.

To show the identification of $\V\otimes \mO_C\cong \mV$ as restricted Lie algebras, it suffices to show the $p$-structure of $\mV$ implies a $p$-structure of $\V$. To do this, one uses again that the $p$-structure of $\mV$ is nothing but an $\mO_C$-linear morphism $F_C^*\mV\rightarrow \mV$. As \[\Hom_{\mO_C}(F_C^*\mV, \mV)\cong \Gamma(C, \mV\otimes F_C^*\mV^\vee)\cong \Gamma(C, \V\otimes_k F_k^*\V^\vee\otimes_k\mO_C)\cong \V\otimes_k F_k^*\V^\vee\cong\Hom_k(F_k^*\V, \V),\]
the Lemma follows.
\end{proof}

In particular, Lemma~\ref{lemma 4} together with the discussions before imply that the kernel $\Ker(\Lie(\phi))$ of $\Lie(\phi)$ corresponds to a constant group scheme $K_0\times C$. And it is nontrivial if $\mbL\neq \Lie(\mX/C)$.

We can now prove of Theorem~\ref{maintheorem 2}.

\begin{proof}[Proof of Theorem~\ref{maintheorem 2}]
  Let $\phi: B_0\times C\rightarrow \mX$ be the induced isogeny from a constant abelian scheme as before. And let  $\Lie(\phi): \Lie(B_0\times C/C)\rightarrow \Lie(\mX/C)$ be the induced map between restricted Lie algebras over $C$ with $\Lie(B_0\times C/C)\cong \mO_C^{\oplus g}$. Suppose that $\mu_{\max}(\Lie(\mX/C))\leq 0$.

  If $\Lie(\phi)=0$, then the discussion in Case I~\ref{case 1} shows that one can has a factorization $\phi=\phi'\circ \Frob^n_{B\times_k C/C}$ with $\phi'$ an isogeny from a constant abelian scheme such that $\Lie(\phi')\neq 0$.

  If $\Lie(\phi')\neq 0$, then the discussion in Case II~\ref{case 2} shows that $\Ker(\Lie(\phi))$
gives a constant group scheme $H\times_{\spec(k)} C\subseteq B_0\times_{\spec(k)} C$. $H$ is nontrivial as $\Lie(\mX/C)$ has degree strictly less than zero. On the other hand, $H\times_{\spec(k)} C=\Ker((B_0\times C)[F]\rightarrow \mX[F])=\Ker(\phi)[F]\subseteq \Ker(\phi)$. Consider the quotient morphism $\pi:B_0\times C\twoheadrightarrow B_0\times C/H\times C\cong (B_0/H)\times C$, $\phi$ would factors as  $\phi=\phi'\circ \varphi$, where $\pi$ is nontrivial and  $\phi'$ is still an isogeny starting from a constant group scheme.

Since there is no terminate to the above two reduction steps, we draw a contradiction as $\phi$ has finite kernel.
\end{proof}

\begin{corollary}
  Let $\phi:C\rightarrow I(x)$ be a nontrivial map from a smooth proper curve to some isogeny leaf, then the Hodge bundle $f_*\Omega_{\mX/C}^1$ would never be nef.
\end{corollary}

\begin{proof}
It follows from Theorem~\ref{maintheorem 2} that $\mu_{\max}(\Lie(\mX/C))>0$. It implies that $\mu_{\min}(f_*\Omega_{\mX/C}^1)<0$, and the corollary follows from the inequality 
  $\bar{\mu}_{\min}(f_*\Omega_{\mX/C}^1)\leq \mu_{\min}(f_*\Omega_{\mX/C}^1)$.
\end{proof}

\begin{remark}
  We remark that the above corollary together with Theorem~\ref{trivial l-adic monodromy} shows that the case (1) of Theorem 1.3 in \cite{Yuan_2021} will only appear when $A/S$ is constant.
\end{remark}



Finally, we can give an estimate of the $l$-adic monodromy group assuming $W_2(k)$-lift.

\begin{proposition}\label{3}
  Let $\varphi: C\rightarrow \mathcal{A}_g$ be a nontrivial map from a smooth proper irreducible curve $C$. Suppose $\varphi$ can be lifted to $W_2(k)$, then the connected algebraic $l$-adic monodromy group $G_l(C)^\circ$ is a connected non-commutative reductive group. In particular, the semisimple group $G_l(C)^{\circ\ad}$ is nontrivial.

 \end{proposition}

\begin{proof}
  Since $C$ is smooth, 
  its algebraic $l$-adic monodromy group $G_l(C)$ is the same as the algebraic $l$-adic monodromy group of its generic point. The latter group is reductive by Tate's isogeny theorem, as the Galois representation on the $l$-adic Tate module acted  by $\Gal(K(C)^s/K(C))$ is semisimple.

   Suppose $G_l(C)^\circ$ is commutative. After replacing  an \'etale cover of $C$ if necessary, we could assume $G_l(C)$ is connected. Indeed, let $$\rho_l: \etpi(C, \bar{x})\rightarrow \GL_{2g}(\Z_l)$$ be the $l$-adic monodromy representation. Then $\rho_l(\etpi(C, \bar{x}))\cap G_l(C)^\circ\leq \rho_l(\etpi(C, \bar{x}))$ is a closed subgroup of finite index. Taking the corresponding finite \'etale cover of $C$, we could assume $\rho_l(\etpi(C, \bar{x}))\subset G_l(C)^\circ$ and hence $G_l(C)$ is connected.

In particular, $G_l(C)$ is commutative. By Theorem~\ref{oort} this implies that $\varphi: C\rightarrow \mA_g$ factors through some isogeny leaf as in the proof of Theorem~\ref{trivial l-adic monodromy}. This then contradicts Theorem~\ref{maintheorem}.
\end{proof}

\subsection{Applications to supersingular K3 surfaces}In this section, we deduce a non-liftability result on supersingular K3 surfaces by using the Kuga-Satake construction. At first, we use the Zarkhin's trick to deduce Theorem~\ref{Main theorem}
from the principally polarized case.

\begin{theorem}[Theorem~\ref{Main theorem}]\label{forget polarization}
 Let $C$ be a smooth proper irreducible curve with
  $f: \mX\rightarrow C$ a family of  (not necessarily principally polarized) abelian varieties of dimension $g$ with small $l$-adic local system. Then $f$ cannot lift to $W_2(k)$.
\end{theorem}

\begin{proof}

If $\mX/C$ is principally polarized, then the theorem follows from Theorem~\ref{maintheorem} and Theorem~\ref{trivial l-adic monodromy}. So, we may assume $\mX/C$ is not principally polarized.

Denote $K\coloneq K(C)$ and $X\coloneq \mX_K$. By taking an \'etale cover if necessary, we suppose that the Galois action of $\Gal(K^s/K)$ on $T_lX$ is trivial. 

Moreover, it is known that if an abelian scheme $\mX/S$ over normal base is relatively projective (see \cite{Raynaud1970FaisceauxAS}, Theorem XI 1.4). So, by choosing a relative ample line bundle $\mathcal{L}$ of $\mX/C$, we  obtain a polarization $\varphi\coloneq\varphi_{\mX/C, \mathcal{L}}: \mX\rightarrow \mX^t$.
On the other hand, $T_lX^t\cong (T_lX)^\vee(1)$ as $\Z_l[\Gal(K^s/K)]$-modules. As $\Z_l(1)$ is the Galois modules associated to the cyclotomic representation\[
\Gal(K^s/K)\rightarrow \Aut(\mu_{l^n}(K^s))
\]for each $n$ and $K$ is a function field over $k=\overline{\F}_p$, $T_lX^t$ also has a trivial $\Gal(K^s/K)$-action. 
In particular, both $X[l^\infty](K^s)$ and $X^t[l^\infty](K^s)$ have trivial $\Gal(K^s/K)$-actions.

Apply this to Zarkhin's trick (see \cite{faltings_degeneration_1990}, Chapter 1, Lemma 5.5): $\mathcal{Z}\coloneq \mX^4\times (\mX^{t})^4$ is a  principally polarized abelian scheme of relative dimension $8g$ over $C$. If $\mX/C$ could lift to $W_2(k)$, so could $\mathcal{Z}/C$. However, the Galois action of $\Gal(K^s/K)$ on $\mathcal{Z}(K^s)\cong \mX(K^s)^4\times \mX^t(K^s)^4$ is again trivial.
This finishes the proof.
    
\end{proof}

\begin{corollary}\label{supersingular cannot lift}
     Let $C$ be a smooth proper irreducible curve with
  $f: \mX\rightarrow C$ a family of supersingular  abelian varieties of dimension $g$. Then $f$ cannot lift to $W_2(k)$.
\end{corollary}

\begin{proof}
  This follows directly from Theorem~\ref{forget polarization} and Theorem~\ref{almostproduct}.
\end{proof}

Fix a degree $2d$ and let $\mM_{2d, K}$ be the moduli stack of primitively polarized K3 surfaces of degree $2d$ with a suitable level structure $K$. Over complex numbers, we have a period map\[p:\mM_{2d, K}\rightarrow \Sh_K(G, \mD)\]to some Shimura variety of a GSpin Shimura datum. As GSpin Shimura datums are of Hodge type, it gives an embedding\[\iota: \Sh_{K}(G, \mD)\hookrightarrow \Sh_{K'}(\GSp, \mH^{\pm})\]to a Shimura variety of Siegel type. Such an embdding is called Kuga-Satake construction. Taking the composite map\[\mM_{2d, K}\rightarrow \Sh_{K'}(\GSp, \mH^{\pm}),\]it extends to  canonical integral models (see \cite{Madapusi_Pera_2016} for instance) and hence admits modding $p$ for almost all $p$. We write $\KS(X)$ for  the abelian variety associated to a K3 surface $X$ via the Kuga-Satake correspondence. The Kuga-Satake construction also maps a supersingular K3 surface to a supersingular abelian variety.

\begin{definition}
  A K3 surface $X/k$ is supersingular if the slopes of the Frobenius action on the crystalline cohomology group $H^2_{cris}(X/W(k))$ are all equal to $1$.
\end{definition}


  

\begin{theorem}
  Let $f: \mX\rightarrow C$ be a family of degree-$d$-polarized supersingular K3 surfaces over a smooth proper curve. Then $f$ cannot lift to $W_2(k)$.
\end{theorem}

\begin{proof}
Consider the family of supersingular abelian varieties $\KS(\mX)/C$ given by the Kuga-Satake construction. Since the Kuga-Satake construction is compatible with base change and $\KS(\mX)/C$ does not admit a $W_2(k)$-lift by Corollary~\ref{supersingular cannot lift}, $f: \mX\rightarrow C$ cannot lift to $W_2(k)$ either.
\end{proof}

\section{Positivity of Hodge bundles and isotriviality}
Fix a smooth proper curve $C$ over a perfect $k$ of characteristic $p>0$ and let $\mX/C$ be an abelian scheme. Let $K=K(C)$ and $X=\mX_K$. In this section, we discuss an opposite direction analysis for a result of Yuan and raise some questions for our future interests.

\begin{theorem}[\cite{Yuan_2021}, Theorem 3.9]\label{yuan}
  Let $C$ be a projective and smooth curve over a finite field $k$ of characteristic $p>0$, and $K=K(C)$. Let $A$ be an abelian variety over $K$ with $p$-rank 0. Suppose $A$ has semistable reduction everywhere over $C$. Then the Hodge bundle of $A$ is not ample over $C$.

\end{theorem}

\begin{remark}
    The original assumption $\Tr_{K/k}A=0$  in \cite{Yuan_2021} Theorem 3.9 is not needed, as $A(K)$ is automatically finitely generated when $k$ is a finite field.
\end{remark}

\begin{remark}
    Proposition~\ref{3.7} shows that for $p\gg 0$, the Hodge bundle as in Theorem~\ref{yuan} will have $\mu_{\min}\leq 0$.
\end{remark}

\subsection{Isotriviality implies non-nefness}

Suppose that $\mX$ admits a constant abelian subscheme, i.e., there exists an abelian variety $B$ over $k$ and an injection $B\times_k C\hookrightarrow \mX$. Then Theorem~\ref{restricted lie algebra equivalence} yields an injection of restricted Lie algebras
\[
\Lie(B\times C/C)\hookrightarrow \Lie(\mX/C).
\]
Since $\Lie(B\times C/C)$ is a free vector bundle of slope $0$, it follows that $\mu_{\max}(\Lie(\mX/C))\geq 0$.

By taking the quotient by the maximal constant abelian subscheme, we may assume that $\mX$ admits no constant abelian subscheme. Suppose the trace morphism $\tau:(\Tr_{K/k}X)_K\rightarrow X$ is nontrivial. Then $\Ker(\tau)$ is a nontrivial finite flat group scheme over $C$ of local-local type by Theorem~\ref{trace}.

By the universal property of the trace, $\Lie(\tau)\neq 0$ by the same argument as in Section~4.1. 

If $\mu_{\max}(\Lie(\mX/C))=0$, the same argument in Section~4.2 shows that $\Lie(\tau)$ must be injective again by the universal property of the trace. In particular, one has an injection\[
\tau:(\text{Tr}_{K/k}X)[F]\times C\hookrightarrow \mX[F].
\] On the other hand, $\Ker(\tau)[F]\cong \Ker(\tau:(\text{Tr}_{K/k}X)[F]\times C\hookrightarrow \mX[F])$. Since $\Ker(\tau)$ is nontrivial of local-local type, $\Ker(\tau)[F]\neq 0$ by the same argument as in Lemma~\ref{local-local gpsch}. This gives a contradiction and hence $\mu_{\max}(\Lie(\mX/C))>0$.
We record this as follows.

\begin{proposition}\label{nonnef}
Suppose that the universal family $\mX/C$ has maximal variation, i.e., it does not admit a constant  abelian subscheme. If $\Tr_{K/k}X\neq 0$, then $f_*\Omega_{\mX/C}^1$ is not nef.
\end{proposition}

\begin{corollary}
  Let $C$ be a  smooth proper curve over a finite field  $k$. Let $\mX/C$ be an abelian scheme generically of $p$-rank 0, then $f_*\Omega_{\mX/C}^1$ is  non-ample. Furthermore, if $\mX$ does not admit a constant abelian subscheme and $\Tr_{K/k}X\neq 0$, then $f_*\Omega_{\mX/C}^1$ is non-nef.
\end{corollary}

\begin{proof}

    The first statement follows from Theorem~\ref{yuan}, as  semistability is preserved by field extension (see \cite{Huybrechts_Lehn_2010} Corollary 1.3.8).
  


The last statement follows from Theorem~\ref{nonnef}.
\end{proof}

\subsection{Non-nefness implies isotriviality}In this section, we are trying to give the opposite direction of Proposition~\ref{nonnef}. Our intuition is that the non-nefness is a totally characteristic $p$ phenomenon. As the triviality of monodromy group is also a totally characteristic $p$ phenomenon, which can be fully described by isogeny leaves by Theorem~\ref{trivial l-adic monodromy for general projective variety}. We have the intuition that if the Hodge bundle of $\mX/C$ is non-nef, then it has a factor detected by the isogeny leaf. In other words, it should have a non-trivial $K'/k$-trace for any extension $K'/K$  that is unramified at all places of $C$ where $\mX$ has good reduction. 

We say that $B_0$ gives an isotrivial abelian subscheme of $\mX$ if for some finite \'etale morphism $C'\twoheadrightarrow C$, there is an injection $B_0\times_kC'\rightarrow \mX\times_CC'$. For example, if there exists a homomorphism $\phi: B_0\times C\rightarrow \mX$ with \'etale kernel, then $B_0$ gives an isotrivial abelian subscheme of $\mX$.

\begin{conjecture}\label{conj 1}
  Let $\mX/C$ be an abelian scheme and suppose $\mX$ does not admit an isotrivial abelian subscheme.
  If $\mu_{\min}(f_*\Omega_{\mX/C}^1)<0$, then for some finite \'etale cover $C'\twoheadrightarrow C$ the $K(C')/k$-trace morphism $$\tau_{K(C')/k}: (\Tr_{K(C')/k}\mX_{K(C')})_{K(C')}\rightarrow X_{K(C')}$$ is nonzero and cannot be descended to $k$.
\end{conjecture}

Here, we suppose $\mX/C$ does not admit an isotrivial abelian subscheme. As for the generally case, one can take successive quotients to obtain $\mX/C$ that has maximal variation. 
Indeed, let $\mX/C$ be an abelian scheme and suppose $X_0$ gives an isotrivial abelian subscheme of $\mX$. Let $C'\twoheadrightarrow C$ be a finite \'etale cover such that one obtains a short exact sequence of abelian schemes over $C'$\[
0\rightarrow X_0\times C'\rightarrow \mX_{C'}\rightarrow \mY\rightarrow 0.
\]

Let $\mX'$ denote  $\mX\times_CC'$.
Consider the commutative diagram with rows are short exact sequences of abelian schemes
\[\begin{tikzcd}
	0 & {X_0\times C'} & \mX' & \mY & 0 \\
	0 & {X_0^{(p)}\times C'} & {\mX'^{(p)}} & {\mY^{(p)}} & 0
	\arrow[from=1-1, to=1-2]
	\arrow[from=1-2, to=1-3]
	\arrow["F", from=1-2, to=2-2]
	\arrow[from=1-3, to=1-4]
	\arrow["F", from=1-3, to=2-3]
	\arrow[from=1-4, to=1-5]
	\arrow["F", from=1-4, to=2-4]
	\arrow[from=2-1, to=2-2]
	\arrow[from=2-2, to=2-3]
	\arrow[from=2-3, to=2-4]
	\arrow[from=2-4, to=2-5]
\end{tikzcd}\]

The snake lemma then gives\[
0\rightarrow X_0[F]\times C'\rightarrow \mX'[F]\rightarrow \mY[F]\rightarrow 0,
\]
and hence a short exact sequence of restricted Lie algebras\[
0\rightarrow \Lie(X_0/k)\otimes_k\mO_{C'}\rightarrow \Lie(\mX'/C')\rightarrow \Lie(\mY/C')\rightarrow 0.
\]

Suppose $\mu_{\min}(f_*\Omega^1_{\mX/C})<0$. So, $\mu_{\max}(\Lie(\mX'/C'))=\mu_{\max}(\Lie(\mX/C))>0$.
Let $$0=\mF_0\subset\mF_1\subset\cdots\subset\mF_n=\Lie(\mX'/C')$$
be the HN-filtration of $\Lie(\mX'/C')$. And let $i$ be the smallest number such that $\mu(\mF_{i+1}/\mF_i)<0$. Then one has \[
0\rightarrow \Lie(X_0/k)\otimes_k\mO_{C'}\rightarrow \mF_i\rightarrow Q\rightarrow 0
\]for some subsheaf $Q\subset \Lie(\mY/C')$. Let $\widetilde{Q}$ be the saturation of $Q$ in $\Lie(
\mY'/C'
)$.

As $\mu_{\max}(\Lie(\mX'/C'))>0$, $\deg(\mF_i)>0$. So, $\deg(\widetilde{Q})\geq \deg(Q)=\deg(\mF_i)>0$. It follows that $\mu_{\max}(\Lie(\mY/C'))\geq \mu(\widetilde{Q})>0$. By induction on dimensions and up to an \'etale cover of $C$, one obtain an abelian scheme that does not admit an isotrivial abelian subscheme.








\begin{remark}
    Suppose furthermore that  $\mX/C$ has generically $p$-rank 0 and does not admit an isotrivial abelian subscheme. Then together with Theorem~\ref{yuan} and Proposition~\ref{3.7}, Conjecture~\ref{conj 1} implies that if $\Tr_{K'/k}X$ is trivial for any extension $K'/K$  that is unramified at all places of $C$ where $\mX$ has good reduction , then $\mu_{\max}(\Lie(\mX/C))=0$ for $p\gg 0$. This looks surprising and we hope to gain a better understanding of this phenomenon.

\end{remark}

\begin{remark}
    Let $\mX/C$ be an abelian scheme that is generically ordinary and does not admit an isotrivial abelian subscheme. 
Suppose that $\mX/C$ has a non-trivial trace $\tau:\Tr_{K/k}X\times _{\spec(k)}C\rightarrow \mX$.  As $\tau$ has a finite kernel, it follows that $X_0\coloneq \Tr_{K/k}X$ is ordinary over $k$. As $\Ker(\tau)$ is a local-local subgroup of $X_0[p^N]\cong (\Z/p^N\Z)^{\oplus 2g}\oplus \mu_{p^N}^{\oplus 2g}$ for $N\gg 0$ by Theorem~\ref{trace}, it follows that $\tau$ is injective, which is in contradiction with our assumption to $\mX/C$. So, $\mX/C$ has a trivial trace and Conjecture~\ref{conj 1} implies that $\mu_{\min}(f_*\Omega_{\mX/C}^1)\geq 0$. A stronger version $\bar{\mu}_{\min}(f_*\Omega_{\mX/C}^1)\geq 0$ was shown by Rössler (\cite{Rossler2015}, Theorem 1.2). This gives another evidence for Conjecture~\ref{conj 1}.

\end{remark}

\subsection{Trivial subbundle implies isotriviality}Let $\mO_C^{\oplus n}$ be a restricted Lie subbundle of $\Lie(\mX/C)$, i.e., there exists a short exact sequence of restricted Lie algebras over $C$\[0\rightarrow\mO_C^{\oplus n}\rightarrow\Lie(\mX/C)\rightarrow\mF\rightarrow 0.\]
Equivalently, this implies there exists a group $k$-scheme $H$ of height one and dimension $n$, such that $H\times_kC$ is a subgroup scheme of $\mX[F]$. 

\begin{question}\label{conj 2}
 If $H\times_kC$ is a constant group subscheme of $\mX[F]$, then does $\mX$ admit an isotrivial abelian subscheme $B_0$ such that  $B_0[F]\cong H$.
\end{question}

This question is inspired in the proof of \ref{maintheorem 2}, where one already constructed such a constant group subscheme of $\mX[F]$ (see the remark after Lemma~\ref{lemma 3}). We used the argument on kernels to show such constant group subscheme has to be full. But it would also be interesting to use reduction on dimension $g$ to have a different proof, for which we raise Question~\ref{conj 2}.

\appendix

\section{Semistable Higgs bundles}\label{appendix}

We would like to include an appendix on the theory of semistable Higgs bundles. Our final goal is to show the category of semistable Higgs bundles of same the slope is an abelian category, for which the author was unabble to find a reference. This is a generalization of the well-known fact that the category of semistable Higgs bundles of same slope is an abelian category.

\begin{definition}[Higgs bundle]
  Let $X$ be a scheme over a field $k$. A Higgs bundle is defined to be a pair $(\mE, \theta)$, where $\mE$ is a locally free sheaf over $X$ and $\theta: \mE\rightarrow \mE\otimes_{\mathcal{O}_X}\Omega^1_X$ is an $\mO_X$-linear map satisfying  the integrality condition $\theta\wedge \theta=0$. We also call such $\theta$ a Higgs field.
\end{definition}

From now,  fix a smooth proper curve $C$ over a field $k$. Then the condition $\theta\wedge\theta$ will be automatic for any linear map $\theta: \mE\rightarrow\mE\otimes \Omega_C^1$, as $\Omega_C^2=0.$

Let $(\mE_1, \theta_1), (\mE_2, \theta_2)$ be two Higgs bundles on $C$, a morphism $f: (\mE_1, \theta_1)\rightarrow (\mE_2, \theta_2)$ is defined to be a $\mO_C$-linear morphism (of sheaves) $f:\mE_1\rightarrow \mE_2$ such that the following diagram commutes

\[\begin{tikzcd}
	{\mE_1} && {\mE_2} \\
	\\
	{\mE_1\otimes \Omega^1_C} && {\mE_2\otimes \Omega_C^1}
	\arrow["f", from=1-1, to=1-3]
	\arrow["{\theta_1}"', from=1-1, to=3-1]
	\arrow["{\theta_2}", from=1-3, to=3-3]
	\arrow["{f\otimes 1}"', from=3-1, to=3-3]
\end{tikzcd}\]

We now write $\text{Hig}_{C}$ for the category of Higgs bundles over $C/k$.

\begin{lemma}\label{torsion-free is locally free}
  A coherent $\mO_C$-module is locally free if and only if it is torsion-free.
\end{lemma}

\begin{proof}
  Since $\dim C=1$ and $C$ is integral smooth, this lemma follows from \cite[\href{https://stacks.math.columbia.edu/tag/0CC4}{Tag 0CC4}]{stacks-project}.
\end{proof}

Let $f: (\mE_1, \theta_1)\rightarrow (\mE_2, \theta_2)$ be a morphism of Higgs bundles. Let $\Ker(f)$ be the kernel of the morphism of sheaves $f:\mE_1\rightarrow \mE_2$. $\Ker(f)$ is torsion-free as being a subsheaf of $\mE_1$, and hence is locally free by Lemma~\ref{torsion-free is locally free}. Moreover, since $\Omega_C^1$ is a  locally free sheaf, we see $\Ker(f\otimes 1)\cong \Ker(f)\otimes 1$. In particular, $\theta_1: \Ker(f)\rightarrow \Ker(f)\otimes \Omega_C$ is a well-defined morphism and $(\Ker(f), \theta_1)$ gives a Higgs bundle. 

Similarly, define $\Image(f)\coloneq \Image(f: \mE_1\rightarrow \mE_2)$. It is torsion-free sheaf as being a subsheaf of $\mE_2$. $\Image(f)$ is then a locally free sheaf by \ref{torsion-free is locally free} again. It is preserved by $\theta_2$ as $\Image(f\otimes 1)=\Image(f)\otimes 1$. In particular, $(\Image(f), \theta_2)$ is a Higgs bundle.

Let $(\mE, \theta)$ be a Higgs bundle. Define a Higgs subsheaf $(\mF, \theta)\subset (\mE, \theta)$ to be a pair consisting of a locally free subsheaf $\mF\subseteq \mE$ preserved by the Higgs field $\theta$.\\

It is known that the Picard group $\Pic(C)$ admits a surjective homomorphism $\deg: \Pic(C)\rightarrow\Z$. Let $\mE$ be a vector bundle (a locally free sheaf) over $C$. Define its degree to be $\deg(\mE)\coloneq \deg(\det(\mE))$ and the slope to be $$\mu(\mE)\coloneq \frac{\deg(\mE)}{\rank(E)}.$$

\begin{definition}Let $(\mE, \theta)$ be a Higgs bundle over $C$. 
  
  1. We say $(\mE, \theta)$ is semistable if for every sub-Higgs bundle $(\mF, \theta)\subseteq (\mE, \theta)$, i.e., $\mF\subseteq \mE$ is a locally free subsheaf preserving the Higgs field $\theta$. Then one has \[\mu(\mF)\leq \mu(\mE).\]

  2. We say $(\mE, \theta)$ is stable if for every sub-Higgs bundle $(\mF, \theta)\subseteq (\mE, \theta)$, i.e., $\mF\subseteq \mE$ is a locally free subsheaf preserving the Higgs field $\theta$. Then one has \[\mu(\mF)< \mu(\mE).\]

  3.  We say $(\mE, \theta)$ is polystable if $(\mE, \theta)$ is a direct sum of stable Higgs bundles of the same slope.
\end{definition}

In particular,  if $\mE$ is (semi)-stable in the classical sense, $(\mE, 0)$ is a (semi)-stable Higgs bundle.

We note that the testing objects for (semi)-stability are subsheaves of $\mE$. Some authors prefer to restrict on subbundles $\mF\subseteq \mE$, i.e., it also has torsion-free cokernel. We are claiming these two notions coincide for the curve case. One direction is obvious. Conversely, we need to show if $(\mE, \theta)$ is a Higgs bundle such that for any subbundle $\mF\subset \mE$ preserved by $\theta$, then $\mu(\mF)<\mu(\mE)$ (resp. $\mu(\mF)\leq\mu(\mE)$). Then, for any locally free subsheaf $\mF\subseteq \mE$ preserved by $\theta$, one still  has $\mu(\mF)<\mu(\mE)$ (resp. $\mu(\mF)\leq \mu(\mE)$). In order to do this, we want to recall the saturation of a (Higgs) bundle.

\begin{lemma}\label{saturation}
  Let $\mE$ be a vector bundle over $C$, and let $\mF$ be a coherent subsheaf of $\mE$. Then 
  
  1. $\mF$ is a vector bundle, and $\mF$ is contained in a subbundle $\widetilde{\mF}$ of $\mE$ with $\rank(\mF)=\rank(\widetilde{\mF})$ and $\deg(\mF)\leq \deg(\widetilde{\mF})$.

  2. If $(\mE, \theta)$ is a Higgs field with $\mF$ preserved by $\theta$, then $\widetilde{\mF}$ is also preserved by $\theta$.
\end{lemma}

\begin{proof}
  1.  Since $\mF$ is torsion-free as being a subsheaf of $\mE$, it is locally free by Lemma~\ref{torsion-free is locally free}. Let $\mathcal{T}$ be the torsion part of $\mE/\mF$ and take $\widetilde{\mF}$ to be the preimage of $\mathcal{T}$ via the quotient map $\mE\twoheadrightarrow \mE/\mF$. $\widetilde{\mF}$ is again a vector bundle with $\mE/\widetilde{\mF}$ being torsion-free by construction. By Lemma~\ref{torsion-free is locally free} again, $\widetilde{\mF}$ is then a subbundle of $\mE$ containing $\mF$, with $\widetilde{\mF}/\mF\cong \mathcal{T}$ is torsion. As $\mathcal{T}$ is torsion, it has rank 0. So, $\rank(\widetilde{\mF})=\rank(\mF)$. The embedding $\mF\hookrightarrow\widetilde{\mF}$ also gives an embedding $\det(\mF)\hookrightarrow \det(\widetilde{\mF})$, from where we have $\deg(\mF)\leq\deg(\widetilde{\mF})$.

  2. Consider the composition of maps\[\widetilde{\mF}\rightarrow \mE\otimes\Omega^1_C\twoheadrightarrow (\mE/\widetilde{\mF})\otimes\Omega^1_C.\]

  It obviously vanishes $\mF$, and we hence have a $\mO_C$-linear map\[\mathcal{T}\rightarrow (\mE/\widetilde{\mF})\otimes\Omega^1_C.\]

  By construction, $\mathcal{T}$ is torsion and $(\mE/\widetilde{\mF})\otimes\Omega^1_C$ is torsion-free, so such a map must be trivial. As a result, the $\mO_C$-linear map $\widetilde{\mF}\rightarrow (\mE/\widetilde{\mF})\otimes\Omega^1_C$
is trivial, and $\theta: \widetilde{\mF}\rightarrow \widetilde{\mF}\otimes \Omega_C^1$ is well-defined.

\end{proof}

\begin{remark}\label{remark on testing obj on higgs}
    In particular, let $(\mE, \theta)$ be a Higgs bundle such that for any subbundle $\mF\subset \mE$ preserved by $\theta$, then $\mu(\mF)<\mu(\mE)$ (resp. $\mu(\mF)\leq\mu(\mE)$). Let $\mF\subseteq \mE$ be a subsheaf preserved by $\theta$. There then exists a subbundle $\widetilde{\mF}$ of $\mE$ with the Higgs field $\theta$ by Lemma~\ref{saturation}. Moreover, $\mF\subseteq\widetilde{\mF}$ by construction and $\mu(\widetilde{\mF})<\mu(\mE)$ (resp. $\mu(\widetilde{\mF})\leq \mu(\mE)$) by the stability (resp. semi-stability) assumption on $(\mE, \theta)$. So, we could draw the conclusion that the testing objects could just be taken to be subbundles.

\end{remark}

\begin{lemma}\label{slope inequality}
  Let $\mE, \mF, \mG$ be vector bundles over $C$. Assume there exists a short exact sequence\[0\rightarrow\mE\rightarrow\mF\rightarrow\mG\rightarrow 0.\]

  1. We have \[\rank(\mF)=\rank(\mE)+\rank(\mG), \deg(\mF)=\deg(\mE)+\deg(\mG).\]

  2. If $\mE, \mF, \mG$ are all nonzero, then we have \[\min\{\mu(\mE), \mu(\mG)\}\leq\mu(\mF)\leq\max\{\mu(\mE), \mu(\mG)\},\]with either equality if and only if $\mu(\mE)=\mu(\mG)$.
\end{lemma}

\begin{proof}
  The first claim is easy. For the second one, we note that \[\mu(\mF)=\frac{\deg(\mF)}{\rank(\mF)}=\frac{\deg(\mE)+\deg(\mG)}{\rank(\mE)+\rank(\mG)}.\]The second claim should follow from an easy calculation.
\end{proof}

\begin{theorem}\label{slope comparison}
  Let $(\mE_1, \theta_1), (\mE_2, \theta_2)$ be two semistable Higgs bundles. Then $$\Hom_{\text{Hig}_{C}}((\mE_1, \theta_1), (\mE_2, \theta_2))=0$$ if $\mu(\mE_2)<\mu(\mE_1)$.
\end{theorem}

\begin{proof}
  Assume $\mu(\mE_2)<\mu(\mE_1)$ and $f:(\mE_1, \theta_1)\rightarrow (\mE_2, \theta_2)$ is a nonzero morphism. So, $(\Image(f), \theta_2)\subseteq (\mE_2, \theta_2)$ is a nonzero Higgs subsheaf and \[\mu(\Image(f))\leq \mu(\mE_2)<\mu(\mE_1)  \tag{$*$} \] by assumption.

On the other hand, we have a short exact sequence of Higgs bundles\[0\rightarrow (\Ker(f), \theta_1)\rightarrow (\mE_1, \theta_1)\rightarrow (\Image(f), \theta_2)\rightarrow 0.\]

$\Ker(f)$ cannot be zero, otherwise $(\mE_1, \theta)$ will be a Higgs subsheaf of $(\mE_2, \theta_2)$ and contradicts the semistability of $(\mE_2, \theta_2)$. By assumption, $\mu(\Ker(f))\leq \mu(\mE_1)$ and hence $$\mu(\mE_1)\leq \max\{\mu(\Ker(f)), \mu(\Image(f))\}=\mu(\Ker(f))$$ by $(*)$ and Lemma~\ref{slope inequality}. It follows that $\mu(\mE_1)=\mu(\Ker(f))=\mu(\Image(f))$ by Lemma~\ref{slope inequality}, which is in contradiction with $(*)$.

\end{proof}

\begin{corollary}
  Let $\mE_1, \mE_2$ be two semistable vector bundles over $C$. Then $\Hom_{\mO_C}(\mE_1, \mE_2)=0$ if $\mu(\mE_2)<\mu(\mE_1)$.
\end{corollary}

\begin{theorem}[Harder-Narasimhan filtration]
Let $\mE$ be a vector bundle on $C$, then it admits a unique filtration by subbundles\[0=\mE_0\subset\mE_1\subset\cdots\subset\mE_n=\mE
\]such that the successive quotients $\mV_{i}/\mV_{i-1}, 1\leq i\leq n$ are semistable vector bundles over $C$ with\[
\mu(\mE_1/\mE_0)>\cdots>\mu(\mE_n/\mE_{n-1}).
\]This filtration is called the Harder-Narasimhan filtration, which we sometimes also write as HN-filtration.
    
\end{theorem}

Define $\mu_{\max}(\mE)=\mu(\mE_1)$ and $\mu_{\min}(\mE)=\mu(\mE/\mE_{n-1})$. Then one can has a generalization of Theorem~\ref{slope comparison} for vector bundles.

\begin{theorem}\label{slope comparison 2}
    Let $\mE, \mF$ be two vector bundles over $C$ with $\mu_{\min}(\mE)>\mu_{\max}(\mF)$. Then $$\Hom_{\mO_C}(\mE, \mF)=0.$$
\end{theorem}

\begin{proof}
Suppose $\mu_{\min}(\mE)>\mu_{\max}(\mF)$ and $\phi: \mE\rightarrow \mF$ is a nontrivial morphism. Let $$0=\mE_0\subset \mE_1\subset\cdots\subset\mE_n=\mE$$ be the HN-filtration of $\mE$ and $$0=\mF_0\subset \mF_1\subset\cdots\subset\mF_m=\mF$$ be the HN-filtration of $\mF$. Let $i\geq 1$ be the smallest number such that $\phi(\mE)\subset \mF_i$. Then $\phi$ gives a nontrivial morphism $\phi: \mE\rightarrow \mF_i\twoheadrightarrow \mF_i/\mF_{i-1}$, with the target is semistable of slope $\mu(\mF_i/\mF_{i-1})<\mu_{\max}(\mF)$. On the other hand, let $j\geq 1$ be the smallest number such that $\phi:\mE_j\rightarrow \mF_i/\mF_{i-1}$ is nontrivial. It then gives a nontrivial morphism $\phi:\mE_j/\mE_{j-1}\rightarrow \mF_i/\mF_{i-1}$, with $\mE_j/\mE_{j-1}$ semistable of slope $\mu(\mE_j/\mE_{j-1})>\mu_{\min}(\mE)>\mu_{\max}(\mF)>\mu(\mF_i/\mF_{i-1})$. Then we have a contradiction by Theorem~\ref{slope comparison}.
\end{proof}

\begin{remark}
    One should not expect Theorem~\ref{slope comparison 2} to work for Higgs bundles, as the HN filtration is not stable under the Higgs field in general.
\end{remark}

\begin{lemma}\label{dual of semistable}
    Let $(\mE, \theta)$ be a Higgs stable (resp. semistable) vector bundle. Then the Higgs bundle $(\mE^\vee, \theta')$ is still Higgs stable (resp. semistable), where $\theta'$ is given by taking the dual of $\theta$ and then tensoring with $\Omega_{C/k}$ on both sides.
\end{lemma}

\begin{proof}
     Let $(\mE, \theta)$ be a Higgs stable (resp. semistable) vector bundle over $C$. It suffices to show if $(\mF, \theta')\subseteq (\mE^\vee, \theta')$ is a subbundle, then $\mu(\mF)<\mu(\mE^\vee)$ (resp. $\mu(\mF)\leq \mu(\mE^\vee)$) by Remark~\ref{remark on testing obj on higgs}.

As $\mF\subseteq \mE^\vee$ and is a subbundle preserved by $\theta$, $(\mE^\vee/\mF, \theta')$ is also a Higgs bundle and one has a short exact sequence\[
0\rightarrow (\mF, \theta')\rightarrow (\mE^\vee, \theta')\rightarrow (\mE^\vee/\mF, \theta')\rightarrow 0.
\]
Taking dual and tensoring with $\Omega_C$ again, one has a short exact sequence\[
0\rightarrow ((\mE^\vee/\mF)^\vee, \theta)\rightarrow (\mE, \theta)\rightarrow (\mF^\vee, \theta)\rightarrow 0.
\]As $(\mE, \theta)$ is Higgs stable (resp. Higgs semistable), one has $\mu((\mE^\vee/\mF)^\vee)<\mu(\mE)$ (resp. $\mu((\mE^\vee/\mF)^\vee)\leq \mu(\mE)$). So, $\mu(\mE^\vee/\mF)>\mu(\mE)$ (resp. $\mu(\mE^\vee/\mF)\geq \mu(\mE)$).  Lemma~\ref{slope inequality} then implies $\mu(\mF)<\mu(\mE^\vee)$ (resp. $\mu(\mF)\leq \mu(\mE^\vee)$), which finishes the proof.

\end{proof}

\begin{lemma}[Destabilizing submodule]\label{maximal destabilizing subsheaf}
  Let $(\mE, \theta)$ be a Higgs bundle and suppose it is not Higgs semistable. Then there exists a Higgs subsheaf $(\mF_{\max}, \theta)\subset (\mE, \theta)$ such that for any Higgs subsheaf $(\mF, \theta)\subset (\mE, \theta)$, one has $\mu(\mF)\leq \mu(\mF_{\max})$. With equality only if $\mF_{\max}=\mF$. This is called the maximal destabilizing subsheaf, and is unique up to isomorphism.
\end{lemma}

\begin{proof}
  Consider the set \[S=\{\mu(\mF); (\mF, \theta)\subset (\mE, \theta)\text{ is a Higgs subsheaf and }\mu(\mF)>\mu(\mE)\}.\]

  $S$ is not empty as $(\mE, \theta)$ is not Higgs semistable. $S$ is a discrete subset of $\Q$, as $\mF\subseteq\mE$ implies $\deg(\mF)\leq \deg(\mE)$. Let $\max(S)$ be the supremum of $S$ and define \[S'\coloneq \{(\mF, \theta); \mu(\mF)= \max(S)\}\]to be set of Higgs subsheaves reaching the maximum slope. Note that as the saturation $\widetilde{\mF}$ of $\mF$ in $\mE$ satisfies $\mu(\widetilde{\mF})\geq \mu(\mF)$, $S'$ acturally consists only of Higgs subbundles.
$S'$ also has a partial order given by $(\mF, \theta)\leq (\mF', \theta)$ if $\mF\subseteq \mF'$. Let $(\mF_{\max}, \theta)$ be a maximal element of $S'$. We now wish to prove this is unique. If not, let $(\mF', \theta)$ be another maximal element in $S'$. Consider the natural map of Higgs bundles $\phi: (\mF_{\max}\oplus \mF', \theta\oplus\theta)\rightarrow (\mE, \theta)$. It gives a short exact sequence\[0\rightarrow (\mF_{\max}\cap \mF', \theta)\rightarrow (\mF_{\max}\oplus \mF', \theta\oplus\theta)\rightarrow (\Image(\phi),\theta)\rightarrow 0.\]By construction, one has $\mu(\mF_{\max}\oplus\mF')=\mu_{\max}$ and $\mu(\Image(\phi)), \mu(\mF_{\max}\cap \mF')\leq \mu_{\max}.$ Lemma~\ref{slope inequality} then implies that $\mu(\mF_{\max}\cap \mF')=\mu(\Image(\phi))=\mu_{\max}$. On the other hand, $\Image(\phi)$ can be defined as the sheafification of the presheaf $U\mapsto \mF_{\max}(U)+\mF'(U)$. As the sheafification functor is exact, one has $\mF', \mF_{\max}\subseteq \Image(\phi)$. By the maximality of $\mF'$ and $\mF_{\max}$, one must have $\mF'=\Image(\phi)=\mF_{\max}$ and the lemma follows.

\end{proof}

\begin{lemma}\label{etale and semistable}
  Let $(\mE, \theta)$ be a Higgs bundle over $C$. Let $f: D\rightarrow C$ be an \'etale cover of degree $d$ between irreducible smooth proper curves over $k$. Then $(\mE, \theta)$ is Higgs stable (resp. Higgs semistable) if and only if $(f^*\mE,f^*\theta)$ is Higgs stable (resp. Higgs semistable).
\end{lemma}

\begin{proof}
  Since $f$ is \'etale, one has $f^*\Omega_C\cong \Omega_D$ and hence $(f^*\mE,f^*\theta)$ is well-defined. Suppose $(f^*\mE, f^*\theta)$ is Higgs stable (resp. Higgs semistable). Let $(\mF, \theta)\subset (\mE, \theta)$ be a Higgs subbundle. Then it gives $(f^*\mF, f^*\theta)\subset (f^*\mE, f^*\theta)$ as a Higgs subbundle. So, one has $\mu(f^*\mF)< \mu(f^*\mE)$ (resp. $\mu(f^*\mF)\leq \mu(f^*\mE)$). Since for any vector bundle $\mV$ over $C$, $\rank(f^*\mV)=\rank(\mV)$ and $\deg(f^*\mV)=d\cdot\deg(\mV)$, the ``if'' direction follows.

  Conversely, let $C'\rightarrow C$ be a Galois cover such that $C'\rightarrow D$ is also an \'etale cover. Then it suffices to show the ``only if'' direction when $D\rightarrow C $ is a Galois cover, as we already showed the ``if'' direction. Let $G=\Aut(D/C)$. And let $(\mF', \theta')\subset (f^*\mE, f^*\theta)$
be the maximal destabilizing subsheaf. Since the maximal destabilizing subsheaf is unique by Lemma~\ref{maximal destabilizing subsheaf}, $(\mF', \theta')$ is stable under the action of $G$ given by the descent datum of $(f^*\mE, f^*\theta)$. By \'etale descent, there exists a subsheaf $\mF\subset \mE$ over $C$ such that $f^*\mF\cong \mF'$. Moreover, $\theta': \mF'\rightarrow \mF'\otimes f^*\Omega_C$ can also be descended to $\theta: \mF\rightarrow \mF\otimes \Omega_C$ as it is preserved by the descent diagram given by $f^*\theta$. Finally, as $(\mE, \theta)$
is Higgs stable (Higgs semistable), it follows that $\mu(\mF)<\mu(\mE)$ (resp. $\mu(\mF)\leq \mu(\mE)$). The lemma then follows from the calculation that $\mu(f^*\mE)=d\cdot \mu(\mE), \mu(f^*\mF)=d\cdot \mu(\mF)$.

\end{proof}

\begin{lemma}
    Let $\mE$ be a vector bundle over $C$ and let $f: D\rightarrow C$ be an \'etale cover between smooth proper curves. Then,

    1. The HN-filtration of $\mE^\vee$ satisfies $\mu_{\min}(\mE^\vee)=-\mu_{\max}(\mE)$ and $\mu_{\max}(\mE^\vee)=-\mu_{\min}(\mE)$.

    2. If the HN-filtration of $\mE$ is \[
    0=\mE_0\subset\mE_1\subset\cdots\subset\mE_n=\mE,
    \]
    Then the  HN-filtration of $f^*\mE$ is given by \[
    0=f^*\mE_0\subset f^*\mE_1\subset\cdots\subset f^*\mE_n=f^*\mE.
    \]
\end{lemma}

\begin{proof}
1. Let $0=\mE_0\subset\mE_1\subset\cdots\subset\mE_n=\mE$ be the HN-filtration of $\mE$. For each $0\leq i\leq n$, consider the short exact sequence of vector bundles\[0\rightarrow \mE_i\rightarrow\mE\rightarrow Q_i\rightarrow 0,
\]with $Q_i\cong \mE/\mE_i$. Taking the dual and we have\[
0\rightarrow Q_i^\vee\rightarrow \mE^\vee\rightarrow\mE^\vee_i\rightarrow 0.
\] Consider the filtration of $\mE^\vee$\[0=Q_n^\vee \subset Q_{n-1}^\vee\subset\cdots\subset Q_0^\vee=\mE^\vee.\] Then $Q_{i}^\vee/Q_{i+1}^\vee\cong (\mE_{i+1}/\mE_i)^\vee$ is semistable of slope $-\mu(\mE_{i+1}/\mE_i)$ by Lemma~\ref{dual of semistable}. The statement then follows by the inequality\[\mu(Q_i^\vee/Q_{i+1}^\vee)=-\mu(\mE_{i+1}/\mE_i)<-\mu(\mE_i/\mE_{i-1})=\mu(Q_{i-1}^\vee/Q_i^\vee).\]

2. This is obvious by Lemma~\ref{etale and semistable}, as $f^*$ is exact and $\mu(f^*\mE)=\deg(f)\cdot \mu(\mE)$ for any vector bundle $\mE$ over $C$.

\end{proof}

\begin{theorem}\label{higgs semistable abelian category}
  Let $f:(\mE_1, \theta_1)\rightarrow (\mE_2, \theta_2)$ be a morphism of  Higgs semistable bundles of the same slope $\mu$. Then one has
  
  1. The kernel, image and cokernel of $f$ will all be either trivial or Higgs semistable bundles of slope $\mu$.

  2. Let $(\mE_1, \theta_1), (\mE_3, \theta_3)$ be  Higgs semistable bundles of the same slope $\mu$ and suppose we have a short exact sequence of Higgs bundles \[0\rightarrow (\mE_1, \theta_1)\rightarrow (\mE_2, \theta_2)\rightarrow (\mE_3, \theta_3)\rightarrow 0,\]then $(\mE_2, \theta_2)$ is Higgs semistable of slope $\mu$.
\end{theorem}

\begin{proof}
  
1. Consider the short exact sequence of Higgs bundles\[0\rightarrow(\Ker(f), \theta_1)\rightarrow (\mE_1, \theta_1)\rightarrow (\Image(f), \theta_1)\rightarrow 0.\]

Let $\widetilde{\Image(f)}$ be the saturation of $\Image(f)$ in $\mE_2$ and we have a Higgs bundle $(\widetilde{\Image(f)}, \theta_2)$ by Lemma~\ref{saturation}. Consider the short exact sequence\[0\rightarrow (\widetilde{\Image(f)}, \theta_2)\rightarrow (\mE_2, \theta_2)\rightarrow (Q, \theta_2)\rightarrow 0.\]

As $(\mE_1, \theta_1), (\mE_2, \theta_2)$ are both Higgs semistable of slope $\mu$, one has\[\mu(\Ker(f))\leq \mu(\mE_1)=\mu, \]\[\mu(\Image(f))\leq\mu(\widetilde{\Image(f)})\leq \mu(\mE_2)=\mu.\]

By Lemma~\ref{slope inequality}, one must have $\mu(\Ker(f))=\mu(\Image(f))=\mu(\widetilde{\Image(f)})=\mu$. In particular, one has $\Image(f)\cong \widetilde{\Image(f)}$ and hence $\Image(f)$ is a subbundle. It implies the cokernel $(\Coker(f), \theta_2)\cong (Q, \theta_2)$ is a Higgs bundle and Lemma~\ref{slope inequality} again implies that $\mu(Q)=\mu$. $(\Ker(f), \theta_1)$ and $(\Image(f), \theta)$ are Higgs semistable obviously. It remains to show $(\Coker(f), \theta_2)$ is Higgs
semistable. Let $(\mF, \theta_2)\subset (\Coker(f), \theta_2)$ be a Higgs subbundle. Let $\widetilde{\mF}$ be the preimage via the surjection $\mF$ of $\mE_2\twoheadrightarrow \Coker(f)$. $\widetilde{\mF}$ is then preserved by $\theta_2$. In particular, we have a short exact sequence\[0\rightarrow (\Image(f), \theta_2)\rightarrow (\widetilde{\mF}, \theta_2)\rightarrow (\mF, \theta_2)\rightarrow 0.\]

As $(\mE_2, \theta_2)$ is Higgs semistable of slope $\mu$, $\mu(\widetilde{\mF})\leq \mu$. And $\mu(\mF)\leq\mu$ follows from Lemma~\ref{slope inequality}. So, $(\Coker(f), \theta_2)$ is Higgs
semistable.

2. $\mu(\mE_2)=\mu$ follows from Lemma~\ref{slope inequality}. We wish to show $(\mE_2, \theta_2)$ is Higgs semistable. Let $(\mF, \theta)\subset (\mE_2, \theta_2)$ be a Higgs subbundle. Then $(\mF\cap\mE_1, \theta)$ gives a Higgs subsheaf of $(\mF, \theta)$. As there is an injection $\mF/\mF\cap\mE_1\hookrightarrow \mE_2/\mE_1\cong \mE_3$ with target torsion-free, one has $(\mF\cap\mE_1, \theta)$ is a Higgs subbundle of $(\mF, \theta)$. Consider the short exact sequence\[0\rightarrow (\mF\cap \mE_1, \theta)\rightarrow (\mF, \theta)\rightarrow (\mF/\mF\cap\mE_1, \theta_1)\rightarrow 0.\]

As $(\mE_1, \theta_1), (\mE_3, \theta_3)$ are Higgs semistable of slope $\mu$. One has $\mu(\mF)\leq \mu$ and $\mu(\mF/\mF\cap \mE_1)\leq \mu$. By Lemma~\ref{slope inequality}, $\mu(\mF)\leq \mu$ also and it follows that $(\mE_2, \theta_2)$ is Higgs semistable.

\end{proof}

\begin{corollary}
  The category of  Higgs semistable bundles of fixed slope $\mu$ is an abelian category.
\end{corollary}

\begin{corollary}
  A Higgs  polystable bundle is Higgs semistable.
\end{corollary}

\printbibliography

\end{document}